\documentclass[a4paper,10pt]{article}
\usepackage[utf8]{inputenc}
\usepackage{array}
\newcolumntype{P}[1]{>{\centering\arraybackslash}p{#1}}

\usepackage[english]{babel}
\usepackage{nomencl}
\usepackage{algpseudocode}
\usepackage{pifont}
\usepackage{subcaption}
\usepackage{ amssymb, amsthm, mathrsfs, bm, mathtools,hyperref, tabularx,graphicx,algorithm}
\usepackage{comment}

\usepackage[numbers,square, comma, sort&compress]{natbib}

\algdef{SE}[DOWHILE]{Do}{doWhile}{\algorithmicdo}[1]{\algorithmicwhile\ #1}%
\theoremstyle{plain}
\newtheorem{theorem}{Theorem}
\newtheorem{lemma}{Lemma}
\newtheorem{remark}{Remark}
\newtheorem{definition}{Definition}

\usepackage{ctable}

\newtheorem{assumption}{Assumption}

\DeclarePairedDelimiter{\abs}{\lvert}{\rvert}
\DeclarePairedDelimiter{\norm}{\lVert}{\rVert}
\begin{document}

\title{A derivative-free optimization algorithm for the efficient minimization of functions obtained via statistical averaging} 
\author{Pooriya Beyhaghi \and
Ryan Alimo
\and Thomas Bewley}

\maketitle

\begin{abstract}
This paper considers the efficient minimization of the infinite time average of a stationary ergodic process in the space of a handful of design parameters which affect it.  Problems of this class, derived from physical or numerical experiments which are sometimes expensive to perform, are ubiquitous in engineering applications. In such problems, any given function evaluation, determined with finite sampling, is associated with a quantifiable amount of uncertainty, which may be reduced via additional sampling.
The present paper proposes a new optimization algorithm to adjust the amount of sampling associated with each function evaluation, making function evaluations more accurate (and, thus, more expensive), as required, as convergence is approached.  The work
builds on our algorithm for Delaunay-based Derivative-free Optimization via Global Surrogates ($\Delta$-DOGS, see JOGO DOI: 10.1007/s10898-015-0384-2).  The new algorithm, dubbed $\alpha$-DOGS, substantially reduces the overall cost of the optimization process for problems of this important class. Further, under certain well-defined conditions, rigorous proof of convergence to the global minimum of the problem considered is established. 
\end{abstract}

\section{Introduction}

In this paper, the Delaunay-based derivative-free optimization algorithm developed in \cite{alimo2017delaunay,alimo2017optimization,zhao2018active, beyhaghi_2,beyhaghi_1, beyhaghi_3}, dubbed $\Delta$-DOGS, is modified
to minimize an objective function, $f(x): \mathbb{R}^n \rightarrow \mathbb{R}$, which cannot be evaluated directly, but can be approximated to a tuneable degree of precision.  As the precision of any function evaluation is increased, the computational (or experimental) cost of that function evaluation is increased accordingly. An example for this type of objective function is the infinite-time average of a discrete-time ergodic process $g(x,k)$ for $k=1,2,3,\ldots$ such that
\begin{subequations}
\label{eq:ProblemStatement}
\begin{equation}
\label{eq:infinitecost}
f(x)=\lim _{N \rightarrow \infty} \frac{1}{N} \sum_{k=1}^N g(x,k), 
\end{equation}
where, for $k\geq \bar k$, 
$g(x,k)$ is assumed to be statistically stationary. The feasible domain in which the optimal design parameter vector $x\in\mathbb{R}^n$ is sought is a bound constrained domain
\begin{equation}
\label{eq:feasibledomain}
L=\{x | a \le x \le b\} \quad
\textrm{where}\  
a<b \in \mathbb{R}^n.
\end{equation}
\end{subequations}
In practice, the precise numerical determination of $f(x)$ for any $x$ is not possible, as this would require infinite time averaging; $f(x)$ can only be approximated as the average of $g(x,k)$ over some \textit{finite} number of samples $N$.  The truth function $f(x)$ is typically a smooth function of $x$, though it is often nonconvex; computable approximations of $f(x)$, however, are generally \textit{nonsmooth} in $x$, as the truncation error (associated with the fact any approximation of $f(x)$ must be computed with finite $N$) is effectively decorrelated from one approximation of $f(x)$ to the next.

Minimizing
\eqref{eq:infinitecost} within
the feasible domain
\eqref{eq:feasibledomain}
is the subject of interest in host of practical applications, such as the optimization of stiffness and shape parameters \cite{mardsen-2004}, feedback control gains in mechanical systems \cite{bewley-2001} and manufacturing processes involving turbulent flows \cite{sethi-2005}, etc.

One interesting class of global optimization algorithms is the Direct (DIviding RECTangles) method which was developed in \cite{jones1993lipschitzian} for optimizing Lipschitz functions. These methods are extended in \cite{slivkins2011multi} to the case of any possible semi-metric by simultaneously considering the subspaces that can contain the optimum. The methods are deterministic optimization algorithms which was designed for problems with exact function evaluation. Later, these methods were modified to solve for objective functions obtained from noisy measurements in \cite{deng2007extension}, \cite{valko2013stochastic} and \cite{kleinberg2008multi}. 

A second relevant class of global optimization algorithms are branch and bound methods \cite{norkin1998branch}, which partition the search domain, then characterize the most promising partition in which to perform the next function evaluation.  These methods were initially designed for optimizing objective functions in which exact function evaluations were possible; however, the methods in \cite{rulliere2013exploring} modified it to address problems with noisy function evaluations.

Another class of methods are polynomial optimization algorithms \cite{lasserre2015introduction}  which globally solves an optimization algorithm using Lasserre type Moment-SOS relaxations. The methods in \cite{nie2019stochastic} extends these algorithms to address stochastic optimization problems. However, these methods are limited to the problems where $f(x)$ is a polynomial function of $x$.

Derivative-free optimization methods is discussed in \cite{audet2017derivative, sankaran2010method,conn2009introduction,conn2009global}, and indeed appear to be the most promising class of approaches for problems of the present form. These methods are implemented for shape optimization in airfoil design \cite{mardsen-2004}, as well as in online optimization \cite{kleinberg2008multi}. With such methods, only values of the function evaluations themselves are used, and neither a derivative nor its estimate is needed. The best methods of this class strive to keep function evaluations far apart in parameter space until convergence is approached, thereby mitigating somewhat the effect of uncertainty in the function evaluations.  This class of methods generally handles bound constraints quite well, and may be used to globally minimize the function of interest.  
Moreover, some advance algorithms \cite{audet2018mesh,audet2018progressive,amaioua2018efficient} in this class can handle problems with nonlinear constraints.
However, this class of method scales poorly with the dimension of the problem.  The surrogate management framework \cite{booker-1999, sankaran2010method, talgorn2015statistical} and Bayesian algorithms \cite{picheny-2013, Rasmussen-2006, Schonlau-1997, Snoek-2012, quan2013simulation} are amongst the best derivative-free methods available today, and are implemented for minimizing a problem of the form in \eqref{eq:ProblemStatement} in \cite{mardsen-2004, mardsen-20041, srinivas2012information, Talnikar-2014}. 
In this class, the method developed in \cite{Snoek-2012} develops a promising Bayesian approach, in a manner which increases the sampling of new measurements as convergence is approached.  However, this method does not selectively refine existing measurements, which is a key contributor to the efficiency of the algorithm developed herein.  

In this paper, a provably globally convergent (under the appropriate assumptions) new optimization approach is developed for problems of the form given in \eqref{eq:ProblemStatement}.  The structure of the remainder of the paper is as follows:  
Section \ref{sec:rew} briefly reviews the key features of the $\Delta$-DOGS($Z$) algorithm developed in \cite{beyhaghi_3}, upon which the present paper is built. 
Section \ref{sec:description} lays out all of the new elements that compose the new optimization approach, as well as the new algorithm itself, dubbed $\alpha$-DOGS.
Section \ref{sec:analyze} analyzes the convergence properties of the new algorithm, and establishes conditions which are sufficient to guaranty its convergence to the global minimum.
Section \ref{sec:result} applies the new algorithm to a selection of model problems in order to illustrate its behavior. 
Some conclusions are presented in Section \ref{sec:conclusion}.

\section{Delaunay-based optimization coordinated with a grid: $\Delta$-DOGS($Z$)} \label{sec:rew}
This section presents a simplified version of the $\Delta$-DOGS($Z$) algorithm, the full version of which is given as Algorithm 2 of \cite{beyhaghi_3}, where it is analyzed in detail. The $\Delta$-DOGS($Z$) algorithm is a grid-based acceleration of the $\Delta$-DOGS algorithm originally developed in \cite{beyhaghi_1}, and is designed to minimize problems in which precise function evaluations are available, while avoiding an accumulation of unnecessary function evaluations on the boundary of the feasible domain.

The optimization problem considered in this section is the minimization of an objective function $f(x)$, approximations of which are assumed to be available, in the feasible domain $L=\{x | x \in \mathbb{R}^n, a \le x \le b\}$. At each iteration of the simplified $\Delta$-DOGS($Z$) algorithm considered here, a metric based on an interpolation of existing function evaluations, and a model $e(x)$ of the ``remoteness'' of any point $x\in L$ from the available datapoints at that iteration (which in the $\Delta$-DOGS($Z$) algorithm characterizes the uncertainty of the interpolation) is minimized to obtain the location of the next point $x\in L$ at which the function will be evaluated. The interpolation and remoteness functions at iteration $k$ are denoted here by $p^k(x)$ and $e^k(x)$, the latter of which is defined below.  Note that the function $e^k(x)$ is called an ``uncertainty'' function in \cite{beyhaghi_1,beyhaghi_2,beyhaghi_3}; we use the name ``remoteness'' function for the same construction in the present paper, as it plays a slightly different role in the sections that follow.

\vskip0.05in
\begin{definition} 
\label{def:uncert}
Consider $S$ as a set of feasible points which includes the vertices of $L$, and
$\Delta$ as a Delaunay triangulation of $S$. 
Then, for each simplex $\Delta_i \in \Delta$, the local remoteness function is defined as:
\begin{subequations}
\label{eq:defe}
\begin{equation} \label{eq:defe;a}
e_i(x)=R_i^2-\norm{x-Z_i}^2, 
\end{equation} 
where $R_i$ and $Z_i$ are the circumradius and circumcenter of $\Delta_i$.
The global remoteness function $e(x)$ is a piecewise quadratic function which is nonnegative everywhere and goes to zero at the datapoints, and is defined as follows:
\begin{equation} \label{eq:defe;b}
e(x)=e_i(x) \quad \forall x \in \Delta_i.
\end{equation}  
\end{subequations}
\end{definition}

\noindent The remoteness function $e(x)$ has a number of properties which are established in Lemmas [2:5] in \cite{beyhaghi_1}, as listed bellow.
\begin{itemize}
\item[a.] The remoteness function $e(x) \ge 0$ for all $x\in L$, and $e(x)=0$ for all $x\in S$.
\item[b.] The remoteness function $e(x)$ is continuous and Lipschitz. 
\item[c.] The remoteness function $e(x)$ is piecewise quadratic with Hessian of $-2\,I$.
\item[d.] The remoteness function $e(x)$ is equal to the maximum of the local remoteness functions:
\begin{equation} \label{eq:emaxei}
e(x)=\max_{1 \leq i \leq N_s} e_i(x),
\end{equation} 
\end{itemize}
where $N_s$ is the number of simplices.

We now review the definition of the Cartesian grid over the feasible domain, as discussed further in \cite{beyhaghi_3}.

\vskip0.05in
\begin{definition} \label{def:cartgrid}
Taking $N_\ell=2^{\ell}$, the Cartesian grid of level $\ell$ over the feasible domain $L=\{x | a \leq x \leq b\}$, denoted $L_{\ell}$, is defined as follows: 
\begin{equation*}
L_{\ell}= \bigg\{ x | x_j= a_j +\frac{b_j-a_j}{N_{\ell}}\cdot z, \quad z \in \{1,2,\dots,N_\ell \}, \quad j \in \{1,2,\ldots,n\} \bigg\}. 
\end{equation*}
The {\it quantization} of a point $x$ onto the grid $L_\ell$, denoted $x_q^{\ell}$, is a point on the grid $L_\ell$ which has the minimum distance from $x$.  Note that this quantization process might have multiple solutions; any of these solutions is acceptable. 
The maximum quantization error of the grid, $\delta_{L_{\ell}}$, is defined as follows:
\begin{equation} \label{eq:quantization_error}
\delta_{L_{\ell}}= \max_{x \in {L_{\ell}}} \norm{x-x_q^{\ell}}=\frac{\norm{b-a}}{2N_\ell}. 
\end{equation} 
\end{definition} 

\begin{remark}\label{rem:cartgrid}
Three important properties of the Cartesian grid for the present purposes follow.
\begin{itemize}
\item[a.] The grid of level $\ell$ covering the feasible domain $L$ in an $n$ dimensional space has $(N_\ell+1)^n$ grid points.
\item[b.] $\lim_{\ell \rightarrow \infty} \delta_{L_{\ell}}=0$.
\item[c.] If $x^\ell_q$ is a quantization of $x$ onto $L_{\ell}$, then $A_a(x) \subseteq A_a(x^\ell_q)$, 
where $A_a(x)$ is the set of active bound constraints at $x$.
\end{itemize}
\end{remark}


\vskip0.05in
Algorithm \ref{algorithm:deltadogs} presents a strawman form of the $\Delta$-DOGS($Z$) algorithm.  A significant refinement of this algorithm is presented as Algorithm 2 of \cite{beyhaghi_3}, together with its proof of convergence and its implementation on model problems.  The key refinement in \cite{beyhaghi_3} of the strawman algorithm presented here is the identification of two different sets of points in $L$, dubbed $S_E$ (on which function evaluations are performed) and $S_U$ (on which function evaluations are not yet performed); the latter set proves useful in \cite{beyhaghi_3} to regularize the triangulation.  The algorithm proposed in \S \ref{sec:description} below effectively generalizes this notion, of evaluation points at which the function has been evaluated, and support points at which the function has not yet been evaluated, to the quantification and control over the {\it extent} of sampling performed for any given approximation of $f(x)$.

\begin{algorithm}[t!]
\caption{Strawman form of the grid-accelerated Delaunay-Based optimization algorithm, $\Delta$-DOGS($Z$), presented as Algorithm 2 in \cite{beyhaghi_3}, which assumes that precise function evaluations are available.}
\label{algorithm:deltadogs}
\begin{algorithmic}[1]
\State Set $k=0$ and initialize $\ell$. Take the set of {initialization points} $S^0$ as all $2^n$ vertices of the feasible domain $L$, together with any user-supplied points of interest (quantized onto the grid $L_0$), and perform function evaluations at all points in $S^0$.
\State Calculate (or, for $k>0$, update) an appropriate interpolating function $p^k(x)$ through all points in $S^k$.
\State Calculate (or, for $k>0$, update) a Delaunay triangulation $\Delta^k$ over all of the points in $S^k$.
\State Find $z$ as a global minimizer of $s_c^k(x)=p^k(x)- K\, e^k(x)$ in $L$, and take $z_\ell$ as its quantization onto the grid $L_{\ell}$.
\State {If} $z_\ell \notin S^k$, then  
take $S^{k+1}=S^k \cup \{z_\ell\}$
calculate $f(z_\ell)$, increment $k$, and repeat from 2.
\State Otherwise, take $S^{k+1}=S^k$, increment both $k$ and $\ell$, and repeat from 2.
\end{algorithmic}
\end{algorithm} 

\section{Delaunay-based optimization of a time-averaged value: $\alpha$-DOGS}
\label{sec:description}

This section presents the essential elements of the new optimization algorithm, dubbed $\alpha$-DOGS, which is designed to efficiently minimize a function $f(x)$ given by \eqref{eq:infinitecost} within the feasible domain $L$ defined by
\eqref{eq:feasibledomain}.
We begin by introducing some fundamental concepts.

\vskip0.05in
\begin{definition}
Take $S$ as a finite set of points $x_i$, for $i=1,\ldots,M$, at which the function $f(x)$ in \eqref{eq:infinitecost} has been approximated; drawing a parallel to the nomenclature commonly used in estimation theory, we refer to any such approximation of $f(x)$, developed with a finite number of samples $N_i$, as a {\it measurement}, denoted $y_i$:
\begin{equation}
 \label{eq:finitecost}
y_i=y(x_i,N_i)=\frac{1}{N_i} \sum_{k=1}^{N_i} g(x_i,k). 
\end{equation}
Any such measurement $y_i$ has a finite uncertainty associated with it, which can be reduced by increasing $N_i$.
\end{definition}

\vskip0.05in
\begin{remark}\label{rem:transient}
For many problems, there is an initial transient such that, for $k<\bar{k}$, the assumption of stationarity of $g(x,t_k)$ is not valid. 
For such problems, the initial transient in the data can be detected using the approach developed in \cite{awad-2006} and set aside, and the signal considered as stationary therafter. In such problems, to increase the speed of convergence of the statistics, the finite sum used for averaging the samples in \eqref{eq:finitecost} is modified to begin at $\bar{k}$ instead of beginning at 1.
\end{remark}
\vskip0.05in

Since, for $k>\bar{k}$, $g(x,k)$ is statistically stationary, each measurement $y_i$ is an \textit{unbiased estimate} of the corresponding value of $f(x_i)$.  We assume that a model for the \textit{standard deviation} quantifying the uncertainty of this measurement, denoted $\sigma_i=\sigma(x_i,N_i)$, is also available.
Since $g(x,t)$ is a stationary ergodic process, for any point $x_i \in L$,
\begin{equation}
\lim_{N_i \rightarrow \infty} y(x_i,N_i) = f(x_i), \quad
\lim_{N_i \rightarrow \infty} \sigma(x_i,N_i)=0.
\end{equation}
 
\begin{remark} \label{remark:IID}
If a stationary ergodic process $g(x,k)$ at some point $x_i \in L$ 
is {\it independent and identically distributed} (IID), then $\sigma(x_i,N_i)={\sigma(x_i,1)}/{\sqrt{N_i}}$;
otherwise, estimates of $\sigma(x_i,N_i)$ can be developed using standard uncertainty quantification (UQ) procedures, such as those developed in \cite{beran1994, beran1995, jenkins1957, oliver-2014, salesky-2012, beyhaghi2018uncertainty}.
The discrete-time process $g(x,k)$ may often be obtained by sampling a continuous-time process $g(x,t)$
at timesteps $t_k=k h$ for some appropriate sample interval $h$. For sufficiently large $h$, the samples of this continuous-time process $g(x,k)$ are often essentially IID; however, with the appropriate UQ procedures in place, significantly smaller sample intervals $h$ will lead to a given degree of convergence in a substantially shorter period of time $t$, albeit with increased storage.
\end{remark}
\vskip0.05in

\begin{definition} \label{def:strict}
Define $S$ as a set of measurements $y_i$, for $i = 1,2, \dots, M$, at corresponding points $x_i$ and with standard deviation $\sigma_i$. We will call a regression $p(x)$ for this set of measurements a \textit{strict} regression if, for some constant $\beta$,
\begin{equation}
\label{eq:strict}
\abs{p(x_i)-y_i} \le \beta\, \sigma_i, \quad \forall\ 1 \le i \le M.
\end{equation}
\end{definition}

Based on the concepts defined above, Algorithm \ref{algorithm:alphadogs} presents our algorithm to efficiently and globally minimize a function of the form \eqref{eq:infinitecost} within a feasible domain defined by \eqref{eq:feasibledomain}.
At each iteration $k$ of this algorithm, $S^k$ denotes the set of $M$ points $x_i$, for $i = 1,2, \dots, M$, at which measurements have so far been made; for each point $x_i \in S^k$, $y_i=y(x_i,N_i)$ denotes the measured value, $\sigma_i=\sigma(x_i,N_i)$ denotes the uncertainty of this estimate, and $N_i$ quantifies the sampling performed thus far at point $x_i$.  Note that the values of $M$, $N_i$, $y_i$, $\ell$, $\alpha$, and $K$ are all updated from time to time as the iterations proceed, and are thus annotated with a $k$ superscript at various points in the analysis of \S \ref{sec:analyze} to remove ambiguity.  Akin to Algorithm \ref{algorithm:deltadogs}, at iteration $k$, $p^k(x)$ is assumed to be a strict regression (for some value of $\beta$) of the current set of measurements $y_i$, and $\ell$ is the current grid level.

\begin{algorithm}[t!]
\caption{The new optimization algorithm, dubbed $\alpha$-DOGS, for minimizing the function $f(x)$ in \eqref{eq:infinitecost} within the feasible domain $L$ defined in \eqref{eq:feasibledomain}.}
\label{algorithm:alphadogs}
\begin{algorithmic}[1]    
\State Set $k=0$ and initialize the algorithm parameters $\alpha$, $K$, $\gamma$, $\beta$, $\ell$, $N^0$, and $N^\delta$ as discussed in \S \ref{sec:description}.  Take the initial set of $M$ sampled points, $S^0$, as the $2^n$ vertices of the feasible domain $L=\{x| a\leq x\leq b\}$ together with any user-supplied points of interest quantized onto the grid $L_\ell$.  Take $N_i=N^0$ for $i=1,\ldots,M$, and compute an initial measurement $y_i=y(x_i,N^0)$ and corresponding uncertainty $\sigma_i=\sigma(x_i,N^0)$ for each point $x_i \in S^0$.

\State Calculate a strict regression $p^k(x)$ for all $M$ available measurements.

\State Calculate (or, for $k>0$, update) a Delaunay triangulation $\Delta^k$ over all of the points in $S^k$.

\State Determine $x_j$ as the minimizer (and, $j$ as the corresponding index), over all $x_i\in S^k$, of the discrete search function $s_d^k(x_i)$, defined as follows:
                \begin{equation} \label{eq:dis_search}
                        s_d^k(x_i)=  \min\{p^k(x_i), 2\, y_i - p^k(x_i)\} - \alpha^k\,\sigma_i^k \quad
                        \textrm{for}\ i=1,\ldots,M.
                \end{equation}

\State Noting the definition of $e^k(x)$ in \eqref{eq:defe}, determine $z$ as the minimizer, over all $x\in L$, of the continuous search function $s_c^k(x)$, defined as follows:
           \begin{equation} \label{eq:continuous_search}
                        s_c^k(x)=  p^k(x) - K^k\,e^k(x) \quad
                        \textrm{for}\ x\in L.
                \end{equation}
Denote $z_\ell$ as the quantization of $z$ onto the grid $L_\ell$. 

\State If
$s_c^k(z) > s_d^k(x_j)$ and
$N_j < \gamma\,2^{\ell}$ where $N_j$ is the current number of samples taken at $x_j$, {then} take $N^\delta$ additional samples at $x_j$,
update $N_j \leftarrow N_j+N^\delta$,
update the measurement $y_j=y(x_j,N_j)$ and uncertainty $\sigma_j=\sigma(x_j,N_j)$,
increment $k$, and repeat from 2.

\State Otherwise, if $z_\ell \notin S^k$, {then} set
$x_{M+1}=z_\ell$ and
$S^{k+1}=S^k \cup \{x_{M+1}\}$, take $N_{M+1}=N^0$, compute the measurement $y_{M+1}=y(x_{M+1},N^0)$ and uncertainty $\sigma_{M+1}=\sigma(x_{M+1},N^0)$, increment $M$ and $k$, and repeat from 2.

\State Otherwise (i.e., if $z_\ell \in S^k$), adjust the algorithm parameters such that $\alpha^{k+1} \leftarrow \alpha^k+\alpha^\delta$ and $K^{k+1} \leftarrow 2\,K^{k}$, increment both $\ell$ and $k$, and repeat from 2.
\end{algorithmic}
\end{algorithm} 

At each iteration of Algorithm  \ref{algorithm:alphadogs}, there are three possible situations, corresponding to three of the numbered iterations of this algorithm:
\begin{itemize}
\item[(6)] The sampling of an existing measurement is increased.  This is called a \textit{supplemental sampling} iteration.
\item[(7)] A new point is identified, and an initial measurement at this point is added to the dataset.  This is called an \textit{identifying sampling} iteration. 
\item[(8)] The mesh coordinating the problem is refined and the algorithm parameters $\alpha$ and $K$ adjusted.  This is called a \textit{grid refinement} iteration.
\end{itemize}
Figure \ref{fig:alpha_illuster} illustrates supplemental sampling and identifying sampling iterations of Algorithm \ref{algorithm:alphadogs}.

 \begin{figure*}
\centerline{
\begin{subfigure}{.5\textwidth}
  \centering
  \includegraphics[width=1.0\columnwidth]{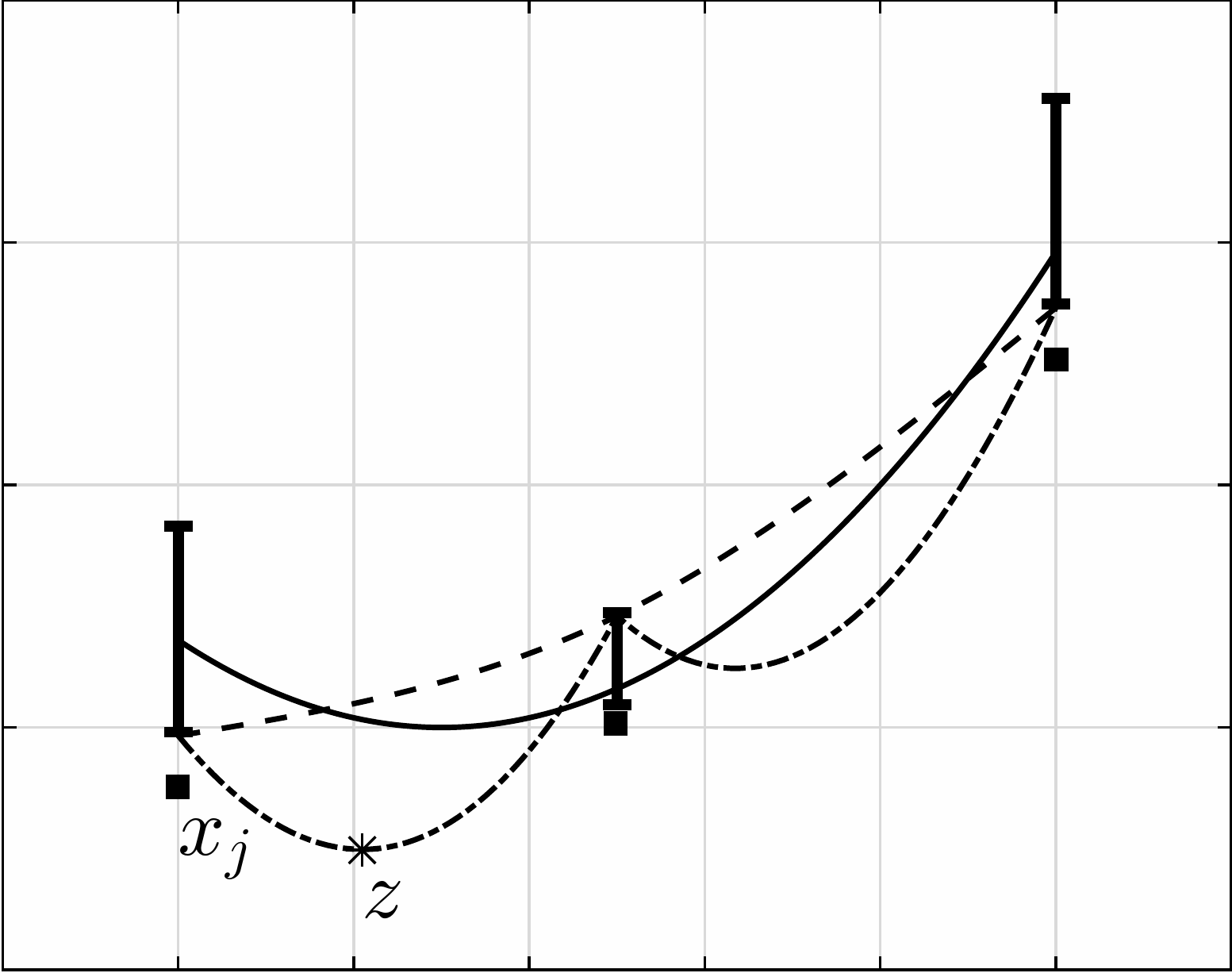} \label{fig:alphaexplore}
  \caption{An identifying sampling iteration}
\end{subfigure} 
\begin{subfigure}{.5\textwidth}
  \centering
  \includegraphics[width=1.0\columnwidth]{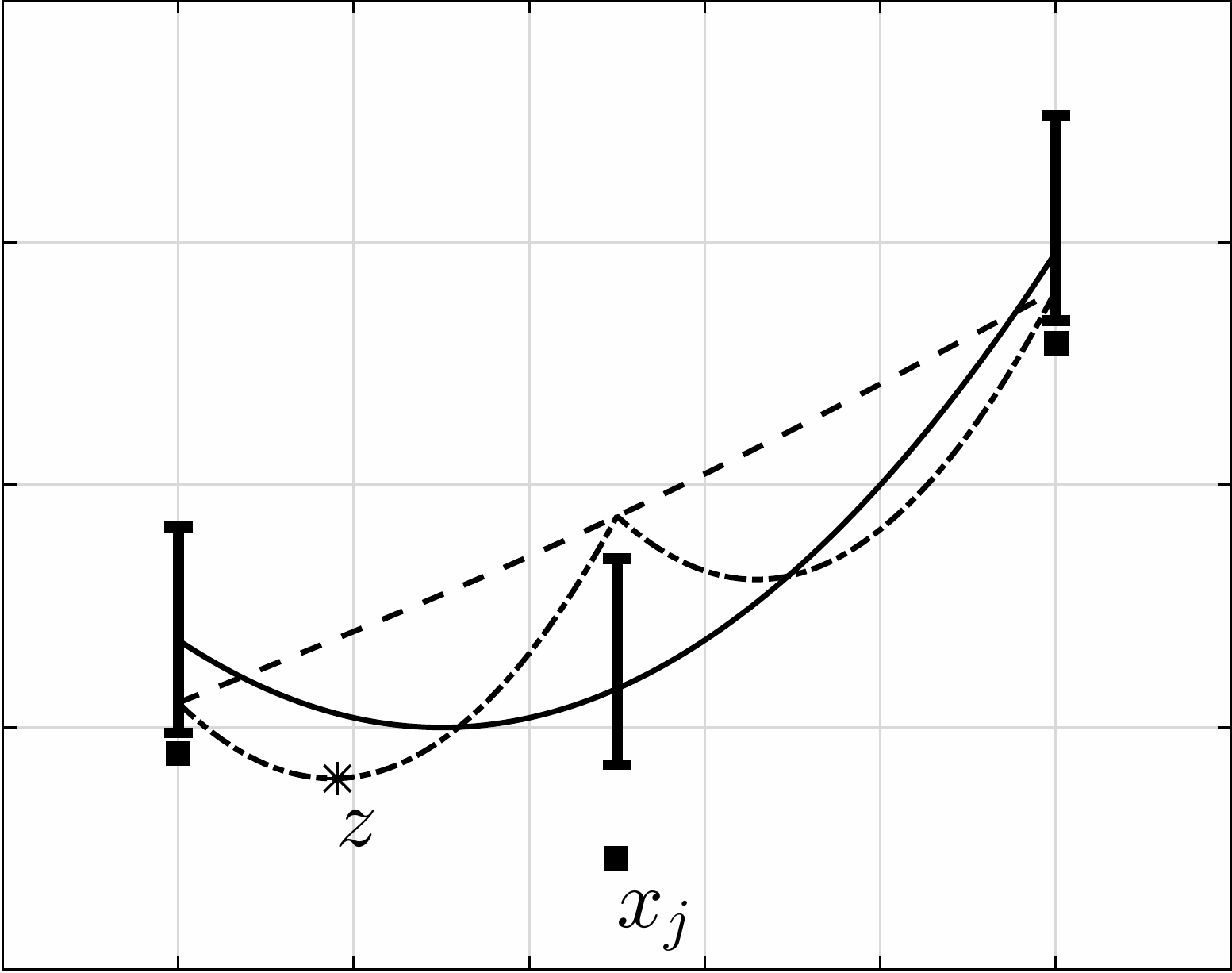} \label{fig:alphaimprove}
  \caption{A supplemental sampling iteration.}
\end{subfigure}
}
\caption{Representation of one iteration in Algorithm \ref{algorithm:alphadogs} in different situations:
(solid line) truth function $f(x)$,
(dashed line) the regression $p^k(x)$,
(dash-dot line) continuous search function $s_c^k(x)$,
(closed squares) $s_d^k(x)$, and
(asterix) $z$.
In figure (a), $s_c^k(z)<s_d^k(x_j)$; it is thus an identifying sampling iteration.
In figure (b), $s_c^k(z)>s_d^k(x_j)$; it is thus an supplemental sampling iteration. Horizontal axis is the $x$ coordinate and the vertical axis is the function value $f(x)$. 
} \label{fig:alpha_illuster}
\end{figure*}

Algorithm \ref{algorithm:alphadogs} depends upon a handful of algorithm parameters, the selection of which affects its rate of convergence, explored in \S \ref{sec:result}, though not its proof of convergence, established in \S \ref{sec:analyze}.  The remainder of this section discusses heuristic strategies to tune these algorithm parameters, noting that this tuning is an application-specific problem, and alternative strategies (based on experiment or intuition) might lead to more rapid convergence for certain problems. 

The first task encountered during the setup of the optimization problem is the definition of the design parameters. Note that the feasible domain considered during the optimization process is characterized by simple upper and lower bounds for each design parameter; normalizing all design parameters to lie between 0 and 1 is often helpful.

The second challenge is to scale the function $f(x)$ itself, such that the range of the normalized function $f(x)$ over the feasible domain $L$ is about unity. If an estimate of the actual range of $f(x)$ is not available a priori, we may estimate it at any given iteration using the available measurements.  Following this approach, at any iteration $k$ with available measurements $\{y_1,y_2,\dots,y_M\}$, all measurements $y_i$, as well as the corresponding uncertainty of these measurements $\sigma_i$, may be scaled by a factor $r_s$ wherever used in iteration $k$ of Algorithm \ref{algorithm:alphadogs} where, for that iteration, $r_s$ is computed such that
\begin{gather*}
r=\frac{1}{\max_{1\le i \le M} \{y_i\}-\min_{1\le i \le M} \{y_i\}}, \quad  r_s =r_l + R(r-r_l) - R(r-r_u), 
\end{gather*}
where $R(x)$ is the ramp function. So defined, $r_s$ is a ``saturated'' version of the factor $r$, constrained to lie in the range $r_l\le r_s \le r_u$.  Note that scaling the $y_i$ and $\sigma_i$ does not interfere with the proof of the convergence of the algorithm, provided in \S \ref{sec:analyze}, but can improve its performance.  In the numerical simulations performed in \S \ref{sec:result}, we take $r_l=10^{-3}$ and $r_u=10^3$.

For problems which are IID with no initial transient, $N^0=N^\delta=1$ is a reasonable starting point; increasing $N^0$ and $N^\delta$ ultimately reduces the number of iterations of the algorithm (and, thus, the number of Delaunay triangulations) required for convergence, but generally increases slightly the total amount of sampling performed.
Suggested values of other algorithm parameters, which work well in the numerical simulations reported in \S \ref{sec:result} but the values of which do not affect the proof of convergence provided in \S \ref{sec:analyze}, include $\alpha^0=\alpha^\delta=0.5$, $K^0=0.5$, $\ell^0=3$, $\beta=4$, and $\gamma=100$.

\section{Analysis of $\alpha$-DOGS}
\label{sec:analyze}

We now analyze the convergence of Algorithm \ref{algorithm:alphadogs}.  We first present some preliminary definitions.

\vskip0.05in
\begin{definition} \label{def:candidatepoint}  
The point $\eta^k \in S$ is called the \textit{candidate} point at iteration $k$ if
\begin{equation}
\eta^k \in \text{argmin}_{z \in S^k} \{y_k(z) + \alpha^k \sigma_k(z)\}.
\end{equation}
Define $f(x^*)$ as the global minimum of $f(x)$ in $L$, then the \textit{regret} is defined as 
\begin{equation} \label{eq:reget}
r_k=f(v_k)-f(x^*). 
\end{equation}
\end{definition}

The definition of the regret given above is common in the optimization literature (see, e.g., \cite{bubeck2011x, kleinberg2008multi, srinivas2012information}).
We show in this section that, under the following assumptions, the regret of the optimization process governed by Algorithm \ref{algorithm:alphadogs} will converge to zero:

\vskip0.05in
\begin{assumption}\label{assumption.1}
A constant $\hat{K}$ exist such that, for all $k>0$ and $x \in L$,
\begin{equation*}
-2\, \hat{K} I\preccurlyeq  \nabla^2 p^k(x) \preccurlyeq  2\, \hat{K} I, \quad -2\, \hat{K} I \preccurlyeq  \nabla^2 f(x) \preccurlyeq 2\, \hat{K}I,
\end{equation*}
\end{assumption}
where $I$ is the identity matrix.

\begin{assumption}\label{assumption.2} 
There is a real continuous and monotonically increasing function $E:\mathbb{R}^{+} \rightarrow [0,Q]$,
which has the following properties:
\begin{itemize}
\item[a.] $E(0) = 0$ and $\lim_{r \rightarrow 0^+} E(r)= 0$ and $\lim_{r \rightarrow \infty} E(r) = Q$.
\item[b.] For all $x\in L$ and $N \in \mathbb{N}$, we have: 
\begin{equation} \label{eq:Tvarp1}
\Big|
\frac{1}{N} \sum_{k=1}^N g(x,k) - f(x)
\Big|
= \abs{y(x,N)-f(x)} \le E(\sigma(x,N)).
\end{equation}
\end{itemize}
\end{assumption}
 
\vskip0.05in
\begin{assumption}\label{assumption.3}
There are real numbers $\alpha>0$ and $\theta \in (0 \, 1]$, such that
\begin{equation}\label{eq:alphabeta}
\sigma(x,N) \le \alpha N^{-\theta} \quad \forall x\in \mathbb{R}, \, N \in \mathbb{N}.
\end{equation}
\end{assumption}

Note that, if the stationary process $g(x,k)$ has a short-range dependence, like ARMA processes, the parameter $\theta=0.5$.  However, for different $\theta$, this general model can also handle stationary processes with long-range dependence, like Fractional ARMA processes. Moreover, at any given point $x \in L$, $\sigma(x,N)$ is a monotonically nonincreasing function of $N$. This condition means that, by increasing the averaging interval, the uncertainty of the estimate must not increase.  

\vskip0.05in
\begin{remark}
Assumption 2 is a stronger condition than ergodicity of $g(x,k)$. Recall that ergodicity of $g(x,k)$ is equivalent to the convergence of $y(x,N)$ to $f(x)$ as $N\rightarrow\infty$, which is the straightforward outcome of Assumptions \ref{assumption.2} and \ref{assumption.3}. 
In Assumption \ref{assumption.2}, the convergence of the sample mean is assumed to be bounded by a function of $\sigma(x,N)$ of the form specified.
\end{remark}

\vskip0.05in
\begin{lemma} \label{lem:epsT1}
For any point $x\in L$ and real positive number $0<\varepsilon<Q$,
\begin{equation} \label{eq:T1}
\abs{y(x,N)-f(x)} - \frac{Q}{E^{-1}(\varepsilon)} \, \sigma(x,N) \le \varepsilon, 
\end{equation} 
\end{lemma}

\begin{proof} 
If $\sigma(x,N)\le E^{-1}(\varepsilon)$ then, since $E(x)$ is an increasing function, by \eqref{eq:Tvarp1}, we have:
\begin{equation*}
E(\sigma(x,N)) \le E(E^{-1}(\varepsilon))= \varepsilon \quad \Rightarrow \quad
\abs{y(x,N)-f(x)}  \leq \varepsilon.
\end{equation*} 
Otherwise, $\sigma(x,N) >E^{-1} (\varepsilon)$; thus, again by \eqref{eq:Tvarp1}, we have
\begin{gather*}
\frac{Q}{E^{-1}(\varepsilon)} \, \sigma(x,N) \geq Q \quad \Rightarrow \quad \abs{y(x,N)-f(x)}
- Q
\le E(\sigma_N)-Q 
\le 0,
\end{gather*} 
Thus, \eqref{eq:T1} is verified for both cases. 
\end{proof}

\begin{lemma}\label{lem:infinite_mesh_decrease}
During the execution of Algorithm \ref{algorithm:alphadogs}, 
there are an infinite number of mesh refinement iterations.
\end{lemma}
\begin{proof}
This lemma is shown by contradiction. 
If Algorithm \ref{algorithm:alphadogs} has a finite number of mesh refinement iterations,
then there is an integer number $\bar\ell$ such that the mesh $L_{\bar{\ell}}$ contains all datapoints obtained by the algorithm.
Since the number of datapoints on this mesh is finite, only a finite number of points must be considered, which leads to having a finite number of identifying sampling iterations.

Since the number of identifying sampling and mesh refinement iterations are finite, there must be an infinite number of supplemental sampling iterations. 
At each supplemental sampling iteration, the averaging length of the estimate at an existing datapoint is incremented by $N^\delta\ge 1$.
Since only a finite number of points is considered, a datapoint exists for which the estimate is improved for an infinite number of supplemental sampling iterations.
As a result, there is an supplemental sampling iteration, such that $N_j>\gamma 2^{\bar\ell}$, which is in contradiction with the assumption of having finite number of mesh refinement iterations.         
\end{proof}

Note that the following short Lemma and proof, which are necessary for this development, are copied directly from \cite{beyhaghi_3}.

\vskip0.05in
\begin{lemma} \label{lem:Gactin}
Consider $G(x)$ as a twice differentiable function such that $\nabla^2 G(x)- 2\, K_1 I \preccurlyeq 0$, 
and $x^* \in L$ as a local minimizer of $G(x)$ in $L$. Then, for each $x \in L$ such that $A_a(x^*) \subseteq A_a(x)$, we have:
\begin{equation} \label{eq:Gxx^*K1}
G(x)-G(x^*) \leq K_1 \norm{x-x^*}^2. 
\end{equation}
\end{lemma}

\begin{proof}
Define function $G_1(x)= G(x)-K_1 \, \norm{x-x^*}^2$. By construction, $G_1(x)$ is concave; therefore,
\begin{gather*}
G_1(x) \le G_1(x^*)+ \nabla G_1(x^*)^T (x-x^*), \\
G_1(x^*)=G(x^*), \quad  \nabla G_1(x^*)= \nabla G(x^*), \\
G(x) \le G(x^*)+ \nabla G(x^*)^T (x-x^*)+ K_1 \, \norm{x-x^*}^2. 
\end{gather*} 
Since the feasible domain is a bounded domain, the constrained qualification holds; therefore, $x^*$ is a KKT point.
Therefore, using $A_a(x^*) \subseteq A_a(x)$ leads to $ \nabla G(x^*)^T (x-x^*)=0$, which verifies \eqref{eq:Gxx^*K1}.
\end{proof}


\begin{lemma}
\label{lem:MKlimsup}
Consider $z$, $x_j$, and $x^*$ as global minimizers of $s_c^k(x)$, $s_d^k(x)$, and $f(x)$, respectively.
Note that $s_d^k(x)$ is only defined for the points in $S^k$, but $s_c^k(x)$ and $f(x)$ are defined over the feasible domain $L$. 
Define $M_k$ as:
\begin{equation} \label{eq:Mk}
M_k=\min \{s_c^k(z)-f(x^*), s_d^k(x_j)-f(x^*)\}.
\end{equation}
Then,
\begin{equation} \label{eq:MK}
\limsup_{k \rightarrow \infty} M_k \leq 0.
\end{equation}
\end{lemma}

\begin{proof}
By Lemma \ref{lem:infinite_mesh_decrease},
there are infinite number of mesh refinement iterations during the execution of Algorithm \ref{algorithm:alphadogs}.  Thus,
\begin{equation}
\lim_{k \rightarrow \infty} K^k = \infty, \quad  \lim_{k \rightarrow \infty} \alpha^k = \infty.
\end{equation}
As a result, for any $0< \varepsilon< Q$, there is a ${k}_{\varepsilon}$ such that, if $k>{k}_{\varepsilon}$, then
\begin{equation} \label{eq:hatkk}
K^k \geq 3\, \hat{K} \quad \text{and} \quad \alpha^k \geq \frac{2\, Q }{E^{-1}(\varepsilon)}. 
\end{equation}
Consider $\Delta^k_{x^*}$ as a simplex in $\Delta^k$, a Delaunay triangulation for $S^k$, that contains $x^*$.
Define $M(x): \Delta^k_{x^*} \rightarrow \mathbb{R}$ as the unique linear function in $\Delta^k_{x^*}$ such that
\begin{equation*}
M(V^k_j)= 2\, f(V^k_j)-p^k(V^k_j), 
\end{equation*} 
where $V^k_j$ are the vertices of $\Delta^k_{x^*}$.
Define $G(x): \Delta^k_{x^*} \rightarrow \mathbb{R}$ as follows:
\begin{equation*} \label{eq:Gdef}
G(x)=s_c^k(x)+M(x)-2\, f(x)=p^k(x)+M(x)-2\, f(x)-K^k\, e^k(x).
\end{equation*} 
By construction, $G(V^k_j)=0$. Moreover:
 \begin{equation*}
\nabla^2 G(x)=\nabla^2 \{p^k(x)-2\, f(x)\}+ 2\, K^k \, I.  
\end{equation*}
Using Assumption \ref{assumption.1} and \eqref{eq:hatkk}, 
$G(x)$ is strictly convex in simplex $\Delta^k_{{x^*}}$.
Since $G(x)=0$ at the vertices of $\Delta^k_{{x^*}}$, then $G(x^*)\leq 0$.
Moreover, since $M(x)$ is a linear function, then
\begin{gather}
\min_{x \in S^k} [2\, f(x)-p^k(x)] \le  \min_{1\le j \leq n+1} \,[2\,f(V^k_j)-p^k(V^k_j)] \le M(x^*), \notag \\ 
s_d^k(x)  \le  [2\, f(x)-p^k(x)] + 2\,(y_k(x)-f(x))-\alpha^k \sigma(x) \label{eq:dkfllf}.
\end{gather}
Using  \eqref{eq:T1} in Lemma \ref{lem:epsT1} and \eqref{eq:hatkk} leads to:
\begin{equation} \label{eq:fkflksigma}
2\, y_k(x)- 2\, f(x)-\alpha^k \sigma(x) \le 2\, \varepsilon. 
\end{equation}
Combining \eqref{eq:dkfllf} and \eqref{eq:fkflksigma} leads to:
\begin{equation*}
s_d^k(x) \le [2\, f(x)-p^k(x)] +2\, \varepsilon.
\end{equation*}
Since $x_j$ is the minimizer of the $s_d^k(x)$, 
\begin{equation*}
s_d^k(x_j) \le \min_{x \in S^k} [2\, f(x)-p^k(x)]+2\,\varepsilon \le M(x^*) +2\, \varepsilon.
\end{equation*}
Furthermore, $z$ is the global minimizer of $s_c^k(x)$ and  $G(x^*) \le 0$; therefore,
\begin{gather} 
s_c^k(z) \le s_c^k(x^*) \le 2\, f(x^*)- M(x^*), \notag  \\ 
s_c^k(z) +s_d^k(x_j) \le 2\, f(x^*) +2\, \varepsilon. \label{eq:scsdkxeps}
\end{gather}
Thus, for any $\varepsilon>0$ and $k> \hat{k}_{\varepsilon}$, \eqref{eq:scsdkxeps} is satisfied;
therefore, \eqref{eq:MK} is verified. 
\end{proof}

\begin{lemma} \label{lem:meshdecreasing}
If $\{k_1,k_2, \dots\}$ are the mesh refinement iterations of Algorithm \ref{algorithm:alphadogs},
then 
\begin{subequations} \label{eq:meshdecreasconvwhole}
\begin{gather} 
\limsup_{i \rightarrow \infty} \Big\{ y(\eta^{k_i},N^{k_i}_{\eta^{k_i}})-f(x^*) + \alpha^{k_i} \sigma(\eta^{k_i},N^{k_i}_{\eta^{k_i}}) \Big\} \leq 0, \quad and  \label{eq:meshdecreasconv} \\
\lim_{i \rightarrow \infty} \sigma(\eta^{k_i},N^{k_i}_{\eta^{k_i}}) =0, 
\label{eq:meshdecreasconvsigma}
\end{gather}
\end{subequations}
where $\eta^{k_i}$ is the candidate point at iteration $k_i$ and $x^*$ is a global minimizer of $f(x)$ in $L$.
\end{lemma}

\begin{proof}
Consider $z$ as a global minimizer of $s^{k_i}_c(x)$ in $L$, and $z_\ell$ as its quantization on $L_\ell$.
Since iteration $k_i$ is a mesh refinement, $z_\ell \in S^{k_i}$.
Consider $\Delta^{k_i}_j$ as a simplex in the Delaunay triangulation $\Delta^{k_i}$ 
which contains $z$. By property (d) of the remoteness function $e^{k_i}(x)$,
\begin{gather*}
e^{k_i}(z_\ell) \ge e^{k_i}_j(z_\ell), \quad e^{k_i}(z) = e^{k_i}_j(z), \\
s^{k_i}_c(z_\ell)- s^{k_i}_c(z) = p^{k_i}(z_\ell)- p^{k_i}(z) +K^{k_i}(e^{k_i}(z)-e^{k_i}(z_\ell)), \\
s^{k_i}_c(z_\ell)- s^{k_i}_c(z) \le p^{k_i}(z_\ell)- p^{k_i}(z) +K^{k_i}(e_j^{k_i}(z)-e_j^{k_i}(z_\ell)). 
\end{gather*}
By Property (c) of Remark \ref{rem:cartgrid}
concerning the Cartesian grid quantizer, 
$A_a(z) \subseteq A_a(z_\ell)$. 
According to Assumption \ref{assumption.1} and Property (c) of the remoteness function introduced in Definition \ref{def:uncert},
$\nabla^2\{p^{k_i}(x) - K^{k_i}e_j^{k_i}(x)\}
- \{\hat{K}+2\, K^{k_i}\} I \le 0$; 
thus, by Lemma \ref{lem:Gactin} and the fact (see Lemma 5 in \cite{beyhaghi_1} for proof) that $z$ globally minimizes $p^{k_i}(x) - K^{k_i}e_j^{k_i}(x)$,   
\begin{equation}
s^{k_i}_c(z_\ell)- s^{k_i}_c(z)
\leq\{\hat{K}+2\, K^{k_i}\} \norm{z_\ell-z}^2.
\end{equation}
Define $\delta_{k_i}$ as the maximum quantization error at iteration $k_i$, then
$\norm{z_\ell-z}\le \delta_{k_i}$.
On the other hand, $z_\ell \in S^{k_i}$, which leads to $s_c^{k_i}(z_\ell)=p^{k_i}(z_\ell)$, and
\begin{equation} \label{eq:zkkhat}
p^{k_i}(z_\ell) \le s_c^{k_i}(z)+ \{\hat{K}+2\, K^{k_i}\} \delta_{k_i}^2.
\end{equation}
At each mesh refinement iteration of Algorithm \ref{algorithm:alphadogs},
there are two possibilities.  
In the first case, $s^{k_i}_c(z) \le s^{k_i}_d(x_j)$;
since $x_j$ is a minimizer of $s_d^{k_i}(x)$, then
\begin{gather} 
p^{k_i}(z_\ell) \le s^{k_i}_d(z_\ell)+ \{\hat{K}+2\, K^{k_i}\} \delta_{k_i}^2, \notag \\
s^{k_i}_d(z_\ell) \le p^{k_i}(z_\ell) - \alpha^{k_i} \sigma(z_\ell,N^{k_i}_{z_\ell}), \notag \\
\sigma(z_\ell,N^{k_i}_{z_\ell}) \le  \frac{\{\hat{K}+2\, K^{k_i}\}}{\alpha^{k_i}} \delta_{k_i}^2. \label{eq:sigmadelta}
\end{gather}
Using \eqref{eq:Mk} (see Lemma \ref{lem:MKlimsup}) and \eqref{eq:zkkhat} leads to
\begin{equation}
p^{k_i}(z_\ell) -f(x^*) \le M_{k_i}+ \{\hat{K}+2\, K^{k_i}\} \delta_{k_i}^2. \label{eq:pkdelta}
\end{equation}
Since the regression is strict,
\begin{equation}\label{eq:explorstrict}
{y(z_\ell,N^{k_i}_{z_\ell})-p^{k_i}(z_\ell)} \le \beta\,
\sigma(z_\ell,N^{k_i}_{z_\ell}).
\end{equation} 
Using \eqref{eq:sigmadelta}, \eqref{eq:pkdelta}, and \eqref{eq:explorstrict} leads to 
\begin{equation} \label{eq:fkexplor}
y(z_\ell,N^{k_i}_{z_\ell}) -f(x^*) \le M_{k_i}+ \{\hat{K}+2\, K^{k_i}\} \delta_{k_i}^2 +\beta \Big[ \frac{\{\hat{K}+2\, K^{k_i}\}}{\alpha^{k_i}} \delta_{k_i}^2\Big]. 
\end{equation}
In the second case, $s^{k_i}_c(z) > s^{k_i}_d(x_j)$, then by the construction of $M_{k_i}$ (see \eqref{eq:Mk}),
\begin{equation}\label{eq:dkimpro}
s^{k_i}_d(x_j) -f(x^*) = M_{k_i}.
\end{equation}
Moreover, since iteration $k_i$ is mesh refinement, then the sampling $N_j\ge \gamma 2^\ell$. 
Thus, using Assumption \ref{assumption.3},
\begin{equation}\label{eq:sigmaimpro}
\sigma(x_j,N^{k_i}_{x_j}) \le \alpha\, \gamma^{-\theta} 2^{-\theta\ell}.  
\end{equation}
Furthermore, the regression $p^{k_i}(x)$ is strict which leads to:
\begin{equation}\label{eq:strictimpro}
y(x_j,N^{k_i}_{x_j})-s^{k_i}_d(x_j) \le
(\beta+\alpha^{k_i})\, \sigma(x_j,N^{k_i}_{x_j}) 
\end{equation} 
Using \eqref{eq:dkimpro}, \eqref{eq:sigmaimpro}, and \eqref{eq:strictimpro} 
leads to
\begin{equation} \label{eq:fkimpro} 
y(x_j,N^{k_i}_{x_j})-f(x^*)
\le M_{k_i}+ (\beta + \alpha^{k_i} ) \alpha \gamma^{-\theta} 2^{-\theta\ell}. 
\end{equation}

Note that $\eta^{k_i}$ is the candidate point at iteration $k_i$.  Thus,
using \eqref{eq:sigmadelta}, \eqref{eq:fkexplor}, \eqref{eq:sigmaimpro}, and \eqref{eq:fkimpro},
and the construction of candidate point (see Definition \ref{def:candidatepoint}), 
\begin{gather}
y(\eta^{k_i},N^{k_i}_{\eta^{k_i}})-f(x^*)+\alpha^{k_i} \sigma(\eta^{k_i},N^{k_i}_{\eta^{k_i}}) \le M_{k_i} + \notag \\
\max \Big \{
(\beta + \alpha^{k_i}) \alpha \gamma^{-\theta} 2^{-\theta\ell},  
(\hat{K}+2\, K^{k_i}) \delta_{k_i}^2 +\beta \Big[ \frac{(\hat{K}+2\, K^{k_i})}{\alpha^{k_i}} \delta_{k_i}^2 \Big] \Big \} \notag \\ 
+ \alpha^{k_i} \max \Big \{
\alpha \gamma^{-\theta} 2^{-\theta\ell}
, 
\frac{(\hat{K}+2\, K^{k_i})}{\alpha^{k_i}} \delta_{k_i}^2
 \Big \}.\label{eq:wki1steq}  
\end{gather}
On the other hand,
\begin{gather}
\delta_{k_i}= \frac{\norm{b-a}}{2^{\ell_0+i}}, \, \quad \alpha^{k_i}= \alpha^0 + i\,\alpha^\delta, \, \quad K^{k_i}= K_0 \, 2^i, \quad \ell^{k_i}= \ell^0+i.  \label{eq:fzkix1}
\end{gather}
By substituting \eqref{eq:fzkix1} in \eqref{eq:wki1steq} and using \eqref{eq:MK} (see Lemma \ref{lem:MKlimsup}),  \eqref{eq:meshdecreasconv} is verified. 
Furthermore, using Assumption \ref{assumption.2}, we have
\begin{equation} \label{eq:yhatfetae}
\abs{y(\eta^{k_i},N^{k_i}_{\eta^{k_i}}) - f(\eta^{k_i})} \le E (\eta^{k_i},N^{k_i}_{\eta^{k_i}}) \le Q.  
\end{equation}
Thus, using \eqref{eq:meshdecreasconv}, $f(\eta^{k_i})-f(x^*)>0$, and \eqref{eq:yhatfetae} leads to
\begin{equation*}
\limsup_{i \rightarrow \infty} \Big\{-Q + \alpha^{k_i} \sigma(\eta^{k_i}, N_{\eta^{k_i}}^{k_i}) \Big\} \le 0.
\end{equation*}
Since $\sigma(\eta^{k_i}, N_{\eta^{k_i}}^{k_i}) \geq 0$ and $ \lim_{i \rightarrow \infty} \alpha^{k_i}= \infty$, \eqref{eq:meshdecreasconvsigma} is verified.


\end{proof}

\begin{theorem}
Consider $\eta^k$ as the candidate point at iteration $k$ of Algorithm \ref{algorithm:alphadogs}, then
\begin{equation}\label{eq:convergence}
\lim_{k \rightarrow \infty} f(\eta^k) =f(x^*),
\end{equation}
where $x^*$ is a global minimizer of $f(x)$.
\begin{proof}

At any iteration $k>k_1$, take $k_i< k$ as the most recent mesh refinement iteration of Algorithm \ref{algorithm:alphadogs}.  Then $\eta^{k_i} \in S^k$, and
\begin{equation}
y(\eta^k, N_{\eta^k}^k)+\alpha^{k} \sigma(\eta^k, N_{\eta^k}^k) \le y(\eta^{k_i}, N_{\eta^{k_i}}^k)+\alpha^{k} \sigma(\eta^{k_i}, N_{\eta^{k_i}}^k).
\end{equation}
Using Assumption \ref{assumption.2} leads to:
\begin{gather*}
\abs{
y(\eta^{k_i}, N_{\eta^{k_i}}^k)-
y(\eta^{k_i}, N_{\eta^{k_i}}^{k_i}) 
}
\le 
E( \sigma(\eta^{k_i},N_{\eta^{k_i}}^{k_i}))+
E( \sigma(\eta^{k_i}, N_{\eta^{k_i}}^k)), \\
y(\eta^k, N_{\eta^k}^k)+\alpha^{k} \sigma(\eta^k, N_{\eta^k}^k) \le  y(\eta^{k_i}, N_{\eta^{k_i}}^{k_i})+
\alpha^{k} \sigma(\eta^{k_i},N_{\eta^{k_i}}^{k_i})+
\notag \\
E( \sigma(\eta^{k_i},N_{\eta^{k_i}}^{k_i}))+
E( \sigma(\eta^{k_i}, N_{\eta^{k_i}}^k)). 
\end{gather*}
By construction, since the sampling at $\eta^{k_i}$ at iteration $k$ is greater than or equal to its sampling at iteration $k_i$,
$\sigma(\eta^{k_i}, N_{\eta^{k_i}}^k) \le \sigma(\eta^{k_i}, N_{\eta^{k_i}}^{k_i})$. 
Since the function $E(x)$ is nondecreasing, 
\begin{equation*}
y(\eta^k, N_{\eta^k}^k)+\alpha^{k} \sigma(\eta^k, N_{\eta^k}^k) \le  y(\eta^{k_i}, N_{\eta^{k_i}}^{k_i})+
\alpha^{k} \sigma(\eta^{k_i},N_{\eta^{k_i}}^{k_i})+
2\, E( \sigma(\eta^{k_i},N_{\eta^{k_i}}^{k_i})). 
\end{equation*}
Using \eqref{eq:meshdecreasconvwhole} in Lemma \ref{lem:meshdecreasing},  Assumption \ref{assumption.2}, and $\alpha^k=\alpha^{k_i}+\alpha^\delta$, leads to:
\begin{equation}
\limsup_{k \rightarrow \infty} y(\eta^k, N_{\eta^k}^k)-f(x^*)+\alpha^k \sigma(\eta^k, N_{\eta^k}^k)\le
\limsup_{k \rightarrow \infty} \big\{ \alpha^\delta \sigma(\eta^{k_i},N_{\eta^{k_i}}^{k_i})
\big\}= 0.
\end{equation}
Similar to the proof of \eqref{eq:meshdecreasconvsigma}, it is thus again easy to show 
\begin{equation*}
    \lim_{i \rightarrow \infty} \sigma(\eta^k, N_{\eta^k}^k)=0.
\end{equation*}
On the other hand, based on Assumption \ref{assumption.2}, and optimality of $f(x^*)$
\begin{gather*}
f(\eta^k)+\alpha^k \sigma(\eta^k, N_{\eta^k}^k) -f(x^*) - E(\sigma(\eta^k, N_{\eta^k}^k)) \le y(\eta^k, N_{\eta^k}^k) +\alpha^k \sigma(\eta^k, N_{\eta^k}^k) -f(x^*), \notag \\
\lim_{k \rightarrow \infty} E(\sigma(\eta^k, N_{\eta^k}^k)) = 0, \\
\limsup_{k \rightarrow \infty} f(\eta^k)-f(x^*) \le 0.
\end{gather*}
Since $f(\eta^k)-f(x^*) \ge 0$, \eqref{eq:convergence} is verified.  
\end{proof}
\end{theorem}

\section{Results}
\label{sec:result}

We now illustrate the performance of Algorithm \ref{algorithm:alphadogs} on some representative examples.  The function $g(x,k)$ in \eqref{eq:infinitecost} is assumed to be a discrete-time statistically stationary random ergodic process.
In this section, we further assume that $g(x,k)$ is IID in the index $k$, and that the variation of $g(x,k)$ from the truth function $f(x)$ is homogeneous in $x$.
In particular, 
\begin{gather*}
g(x,k)=f(x)+v_k \quad
\text{where} \ v_k=\mathcal{N}(0,0.3).
\end{gather*}

In this section, two different test functions for $f(x)$ are considered within the simple feasible domain $L=\{x|0 \le x_i \le 1\ \forall i\}$, the
\textit{shifted parabolic function} 
\begin{equation}\label{eq:parab}
                f(x) = \frac{5}{n} \sum_{i=1}^n (x_i-0.3)^2, 
\end{equation}
with a global minimizer in $L$ of $x^*_i=0.3$ and a corresponding global minimum of $f(x^*)=0$, and the \textit{scaled Schwefel fuction}
\begin{equation}\label{eq:Schwefel}
                f(x) = 0.83797 - \frac{1}{n} \sum_{i=1}^n x_i \sin (500\, \abs{x_i}),
\end{equation}
with a global minimizer in $L$ of $x^*_i=0.8419$ and a corresponding global minimum of $f(x^*)=0$. We will consider these two functions in $n=1$, 2, and 3 dimensions.

One-dimensional representations of these functions are illustrated in Fig \ref{fig:testproblems}:
for the shifted parabolic function \eqref{eq:parab}, the truth function (unknown to the optimization algorithm) is a simple parabola, whereas for the scaled Schwefel fuction \eqref{eq:Schwefel}, the truth function is a smooth nonconvex function with four local minima.
Note that the perturbations present in several measurements of these functions, computed with finite $N_i$, result in a complicated, nonsmooth, nonconvex behavior.  This paper shows how to efficiently minimize such functions based only on such noisy measurements, automatically refining the measurements (by increasing the sampling) as convergence is approached.

\begin{figure*}
\centerline{
\begin{subfigure}[b]{.5\textwidth}
    \centering
        \includegraphics[width=1.1\columnwidth]{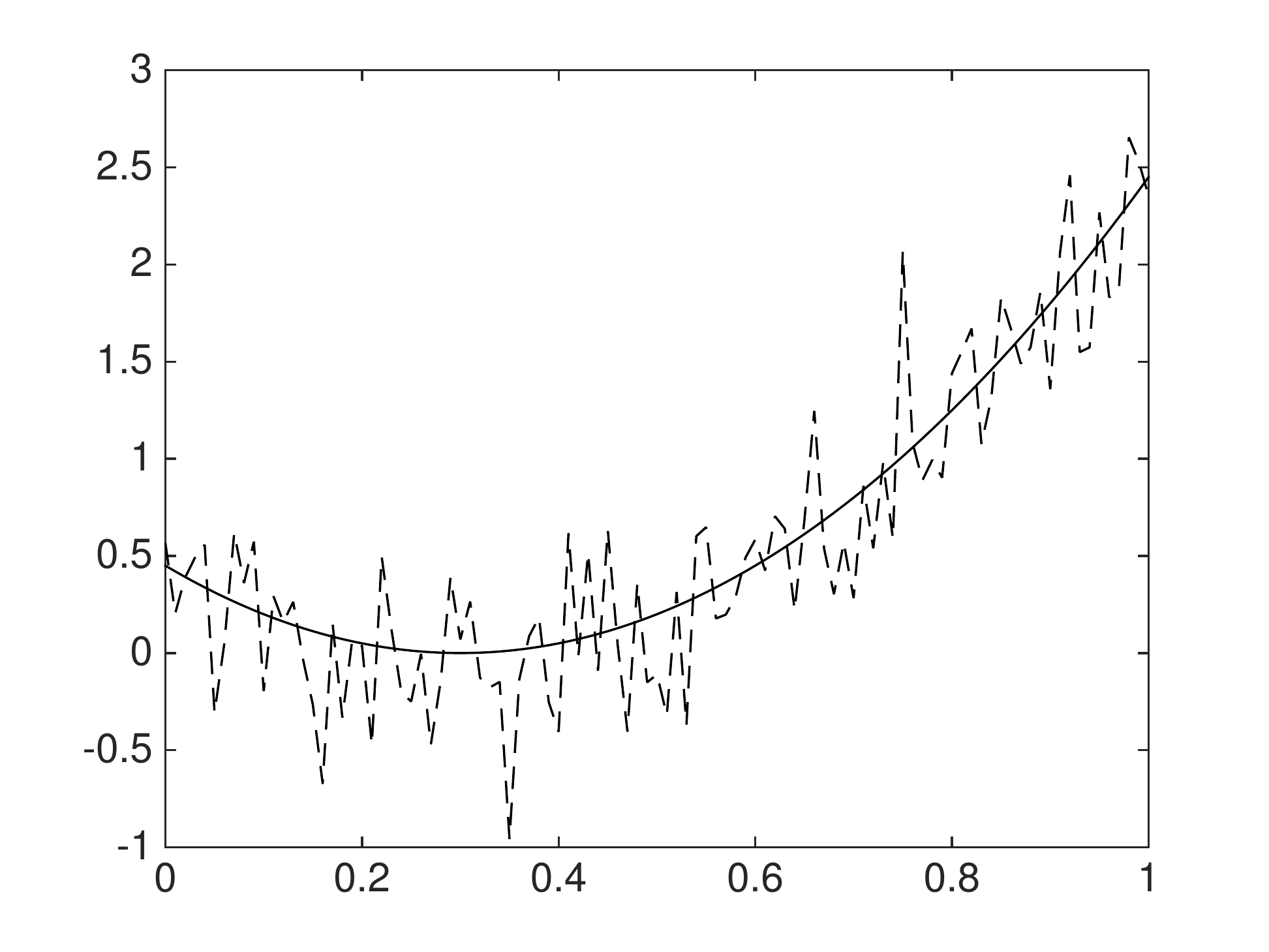}\vskip-0.01in
    \caption{Shifted parabolic function \eqref{eq:parab}.} 
    \end{subfigure}
\begin{subfigure}[b]{.5\textwidth}
  \centering
    \includegraphics[width=1.1\columnwidth]{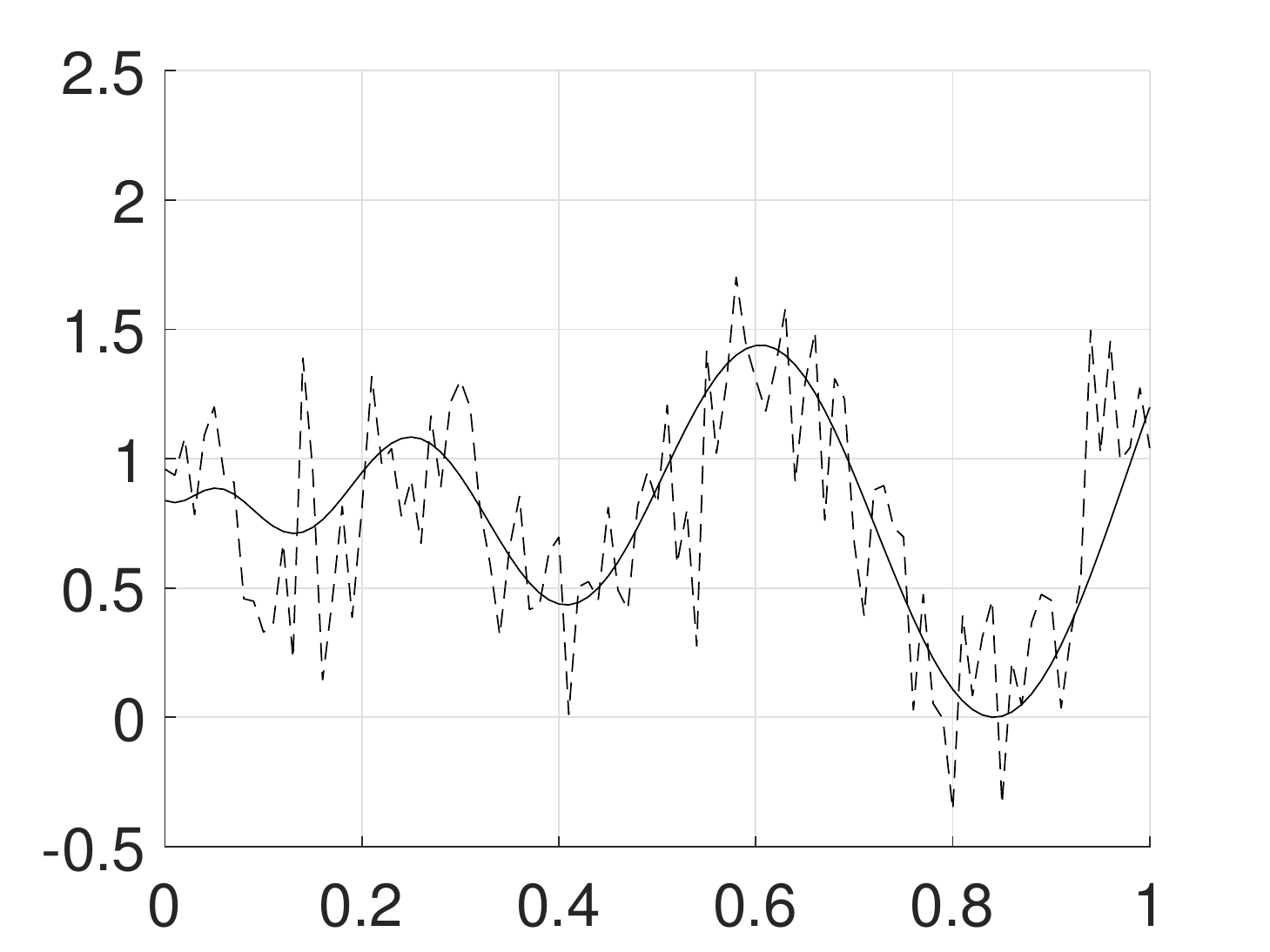}\vskip-0.01in
    \caption{Scaled Schwefel fuction \eqref{eq:Schwefel}.}  
     \end{subfigure}
     }
     \caption{Illustration of test problems \eqref{eq:parab} and \eqref{eq:Schwefel}. 
     (solid line) truth function $f(x)$, and (dashed line) a set of measurements $y_i$ computed with a single sample at each measurement, $N_i=1$.}
     \label{fig:testproblems}
\end{figure*}

The optimizations are initialized with measurements of sample length $N^0=1$ at the vertices of $L$.  Figure \ref{fig:1dilluster_implement} illustrates the application of Algorithm \ref{algorithm:alphadogs} after $k=200$ iterations in the 1D case, taking $N^0=N^\delta=1$ additional sample (at either a new measurement point, or at an existing measurement point) at each iteration of the algorithm.  In Figure
\ref{fig:1dilluster_implement}a, the sampling $N_i$ after $k=100$ iterations (plus the 2 initial sample points, for a total of $202$ samples) at the $M=5$ measured points $y_i$ indicated, enumerated from left to right, is $\{25,94,58,24,1\}$; in Figure
\ref{fig:1dilluster_implement}b, the sampling $N_i$ after $202$ iterations at the $7$ measured points indicated is $\{7,6,11,4,55,115,4\}$.  Both results clearly show that the algorithm focuses the bulk of its sampling in the immediate vicinity of the minimum, where the accuracy of the measurements is especially important, while avoiding unnecessary sampling far from the minimum, where the accuracy of the measurements is of reduced importance. It is also seen that more exploration is performed for the scaled Schwefel function than for the shifted parabolic function, as a result of its more complex underlying trend.

\begin{figure*}
\centerline{
\begin{subfigure}[b]{.5\textwidth}
    \centering
        \includegraphics[width=1.1\columnwidth]{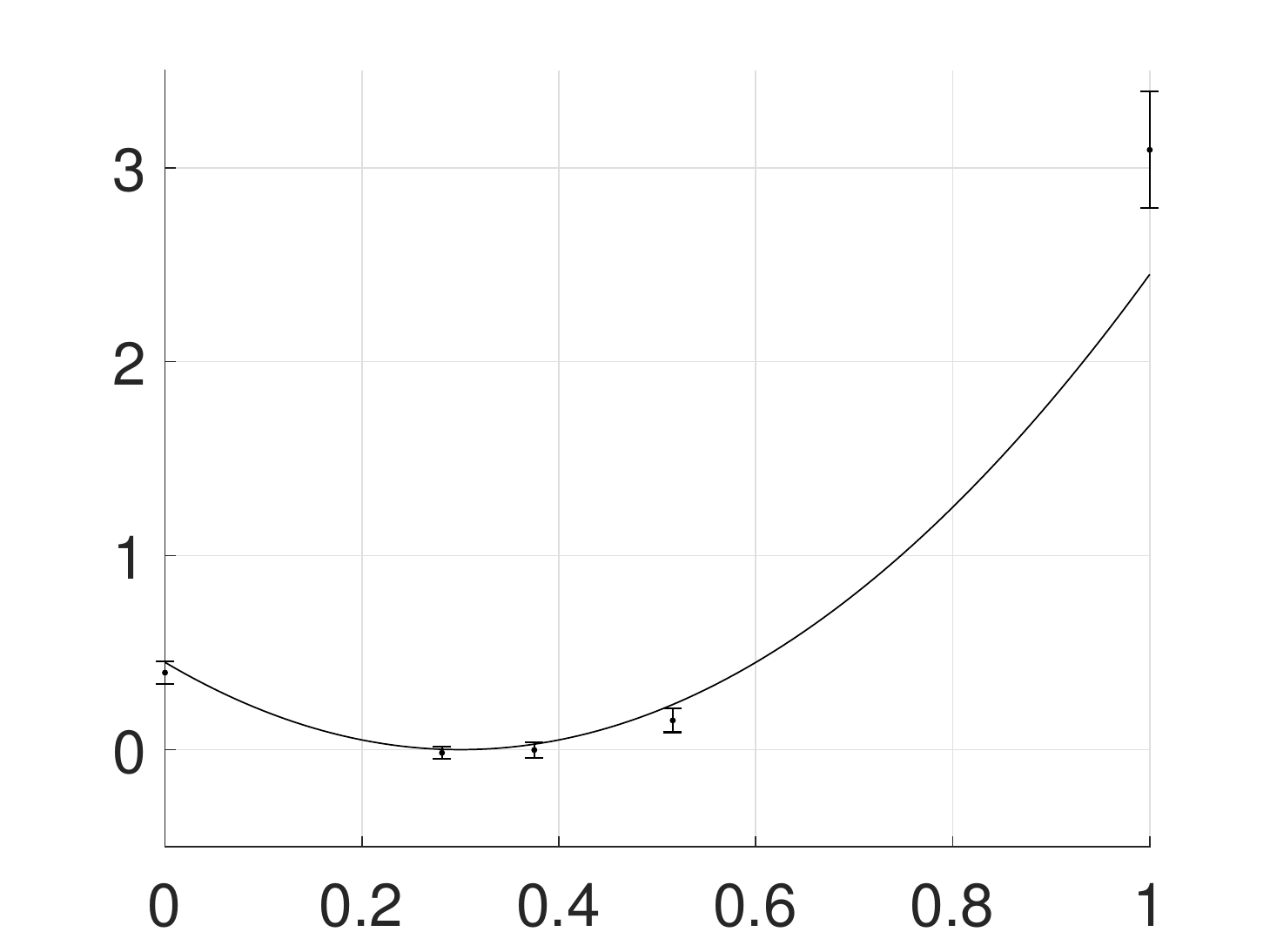}\vskip-0.01in
    \caption{Shifted parabolic function \eqref{eq:parab}.} 
    \end{subfigure}
\begin{subfigure}[b]{.5\textwidth}
  \centering
    \includegraphics[width=1.1\columnwidth]{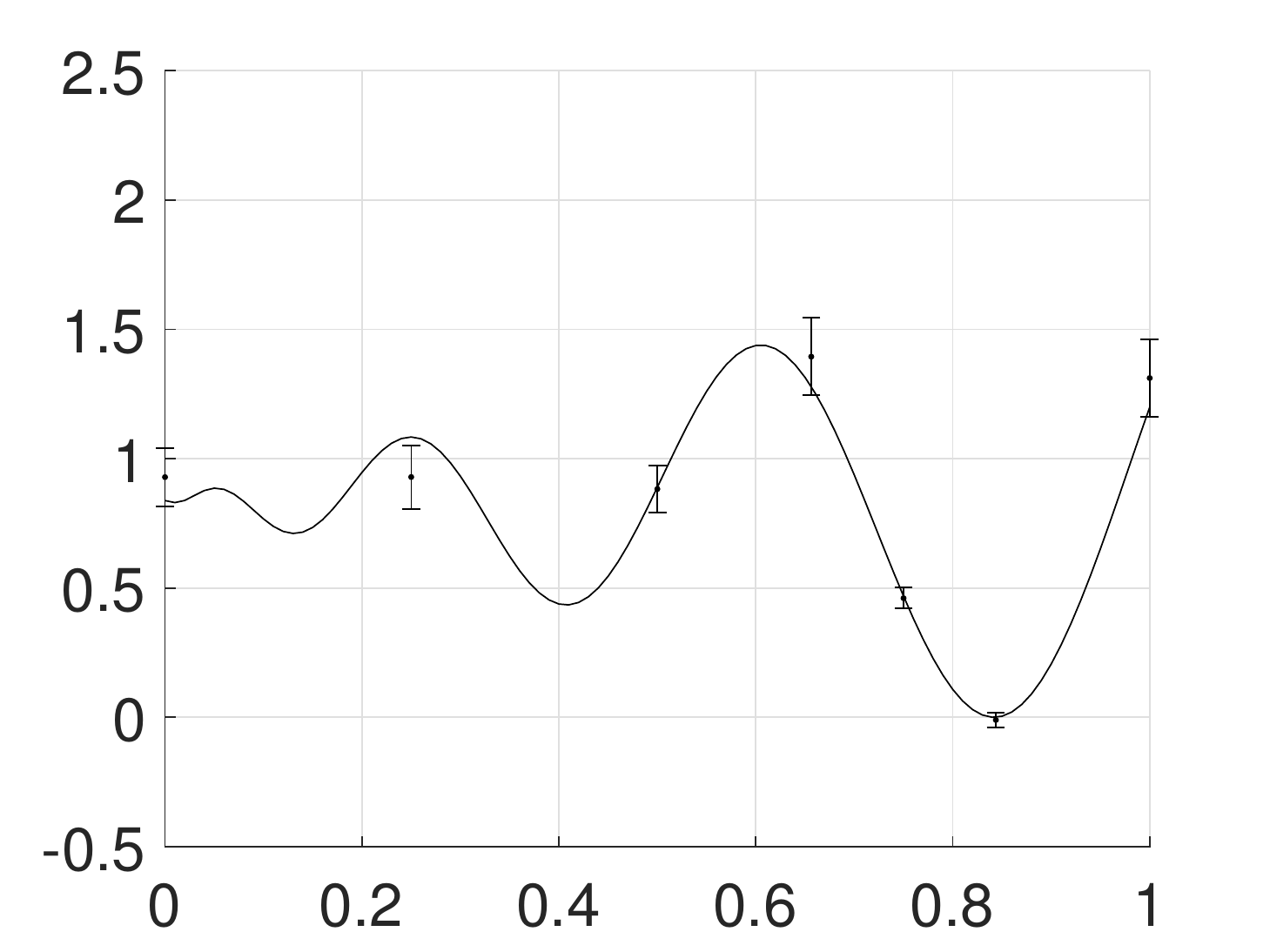}\vskip-0.01in
    \caption{Scaled Schwefel function \eqref{eq:Schwefel}.}  
     \end{subfigure}
     }
     \caption{Illustration of Algorithm \ref{algorithm:alphadogs} on model problems \eqref{eq:parab} and \eqref{eq:Schwefel} in 1D after 100 iterations, taking $N^0=N^\delta=1$: 
     (solid line) the truth function $f(x)$, and (error bars) the 66 percent confidence intervals of the measurements.}
     \label{fig:1dilluster_implement}
\end{figure*}

Since the function evaluation process in these tests has a stochastic component,
Algorithm \ref{algorithm:alphadogs} was next applied an ensemble of twenty separate tests, for both model problems discussed above,
in each of three different cases with increasingly higher dimension (that is, $n=1$, $n=2$, and $n=3$).  The
convergence histories of these simulations are illustrated in Figures \ref{fig:parabhis} and \ref{fig:shwefelbhis}.

To better quantify the performance of the algorithm proposed, we now introduce the following concept.

\vskip0.05in
\begin{definition}
Assume that the stationary process $g(x,k)$ is IID, and that the nominal variance $\sigma(x_i,1)=\sigma_0$ for all points $x_i \in L$.  As mentioned in Remark \ref{remark:IID}, denoting $N_i$ as the total number of samples taken at point $x_i$, the $99.6$ confidence intervals of the corresponding measurement $y_i$ given by \eqref{eq:finitecost} is $3 \, \sigma_i=3 \, \sigma(x_i,N_i)=3 \,\sigma_0/\sqrt{N_i}$.  If we assume that all of sampling of the algorithm is performed at a single point, the uncertainty of this single measurement after $k$ samples would thus be $\sigma_0/\sqrt{k}$, which we refer to as the \textit{reference error}. 
This function is indicated in Figures \ref{fig:parabhis} and \ref{fig:shwefelbhis} by a solid line of slope $-{1}/{2}$ in log-log coordinates.
\end{definition} 

\vskip0.05in
It is observed that, (see Figures \ref{fig:parabhis} and \ref{fig:shwefelbhis}), in which we have again taken 1 new sample at each iteration, the averaged value of the regret function of Algorithm \ref{algorithm:alphadogs} is eventually diminished to a value close to the reference error.  That is, the value of the regret at the end of these optimizations is actually proportional to the uncertainty of a single measurement, assuming that all of the sampling is done at a single point.

Figures \ref{fig:parabhis} and \ref{fig:shwefelbhis} also report the number of datapoints which are considered by the optimization algorithm as the iterations proceed.  This number is important in optimization problems for which the function evaluations are obtained from simulations which have an (expensive) initial transient, which must be set aside before sampling the statistic of interest, as discussed further in Remark \ref{rem:transient}. It is observed, as in the 1D case illustrated in Figure \ref{fig:1dilluster_implement}, that the number of datapoints that are considered for the shifted parabolic function is less than that for the scaled Schwefel function.
Further, the regret function converges faster to the general proximity of the global solution for the shifted parabolic function.
This result is reasonable, since the underlying function in the shifted parabolic case is much simpler.

\begin{figure*}
\centerline{
\begin{subfigure}[b]{.5\textwidth}
    \centering
        \includegraphics[width=1.1\columnwidth]{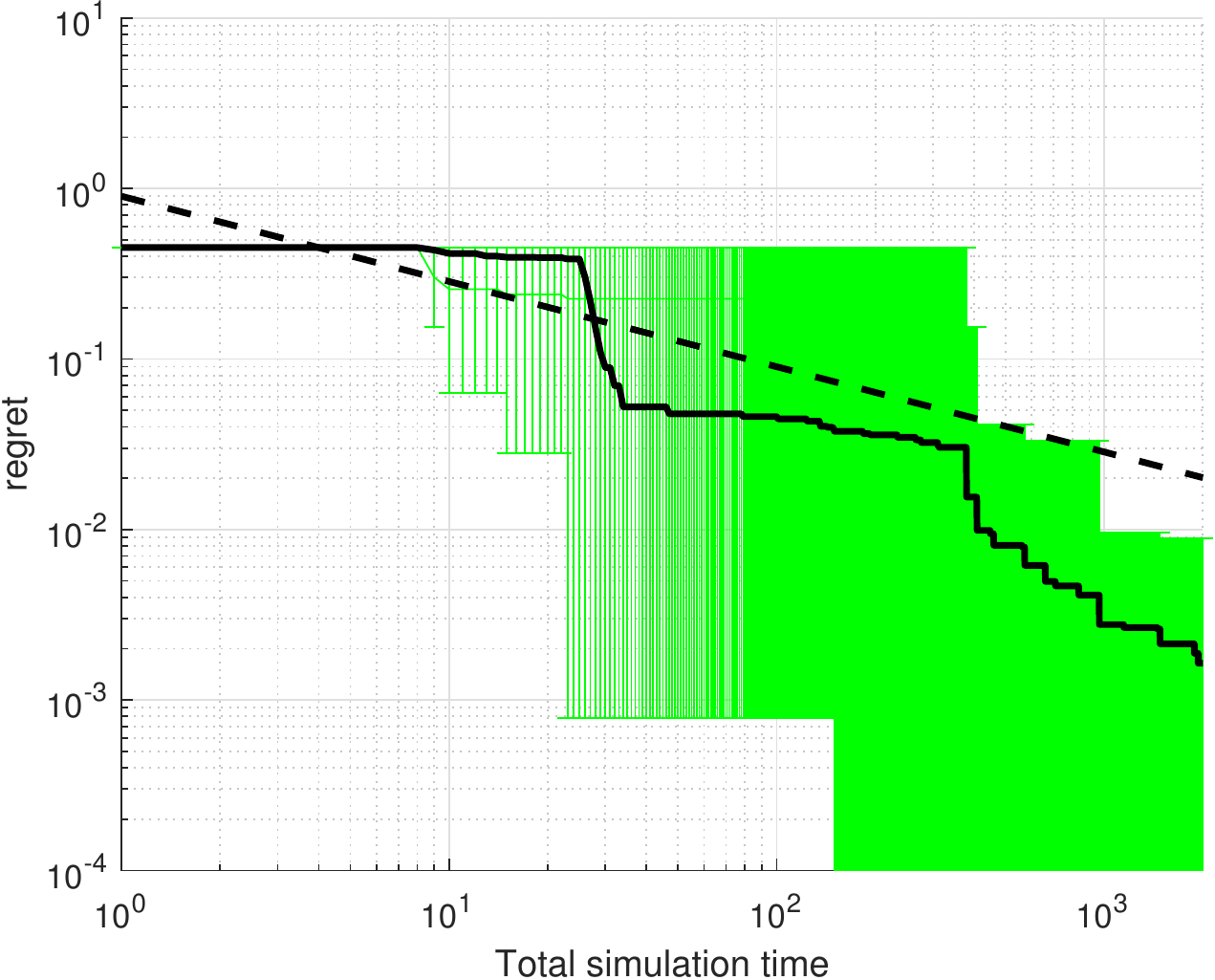}~\vskip-0.01in
    \caption{The regret function in 1D.} 
    \end{subfigure}
~\begin{subfigure}[b]{.5\textwidth}
  \centering
    \includegraphics[width=1.1\columnwidth]{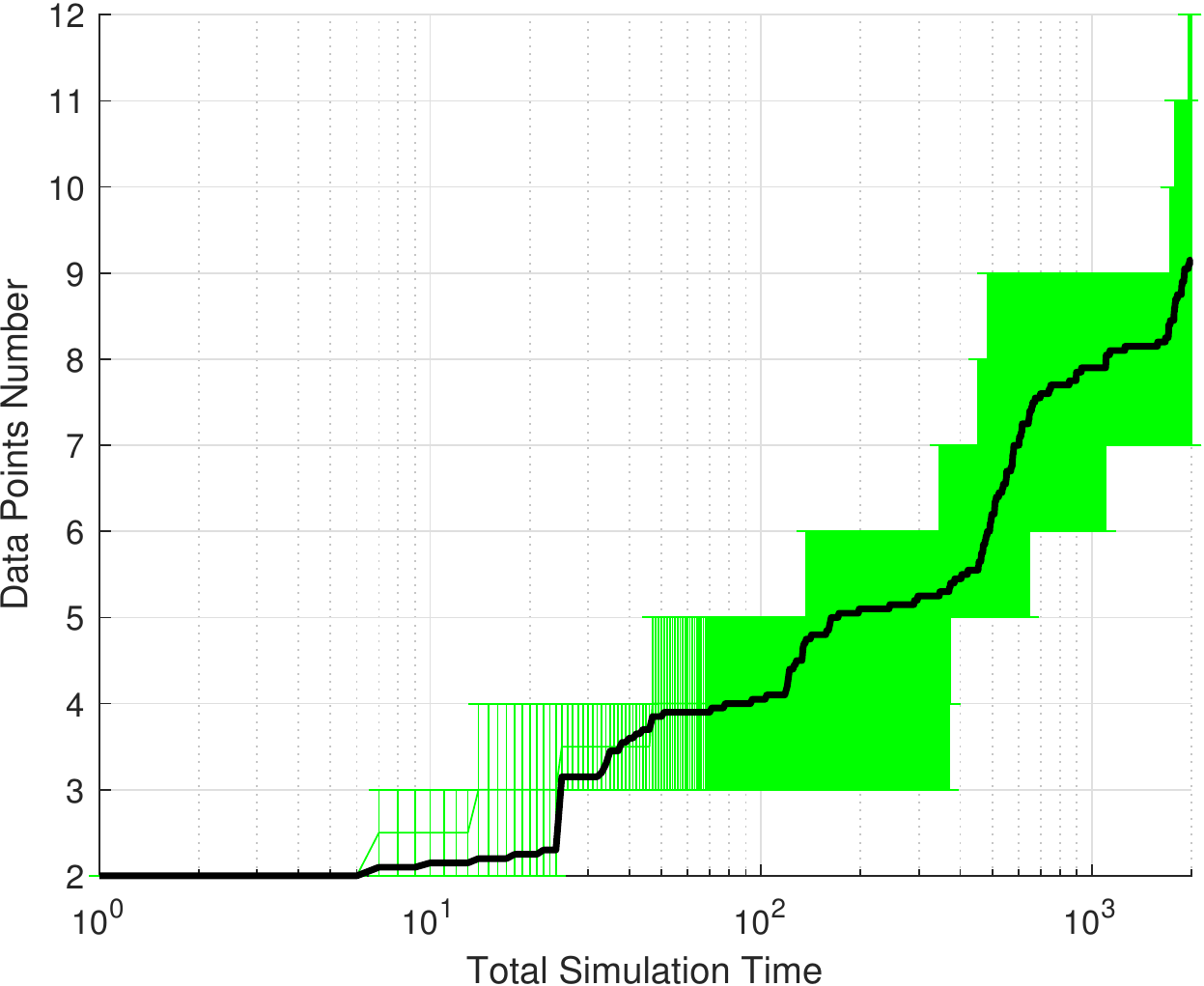}\vskip-0.01in
    \caption{Total number of datapoints in 1D.}  
     \end{subfigure}
     }
 \centerline{
\begin{subfigure}[b]{.5\textwidth}
    \centering
        \includegraphics[width=1.1\columnwidth]{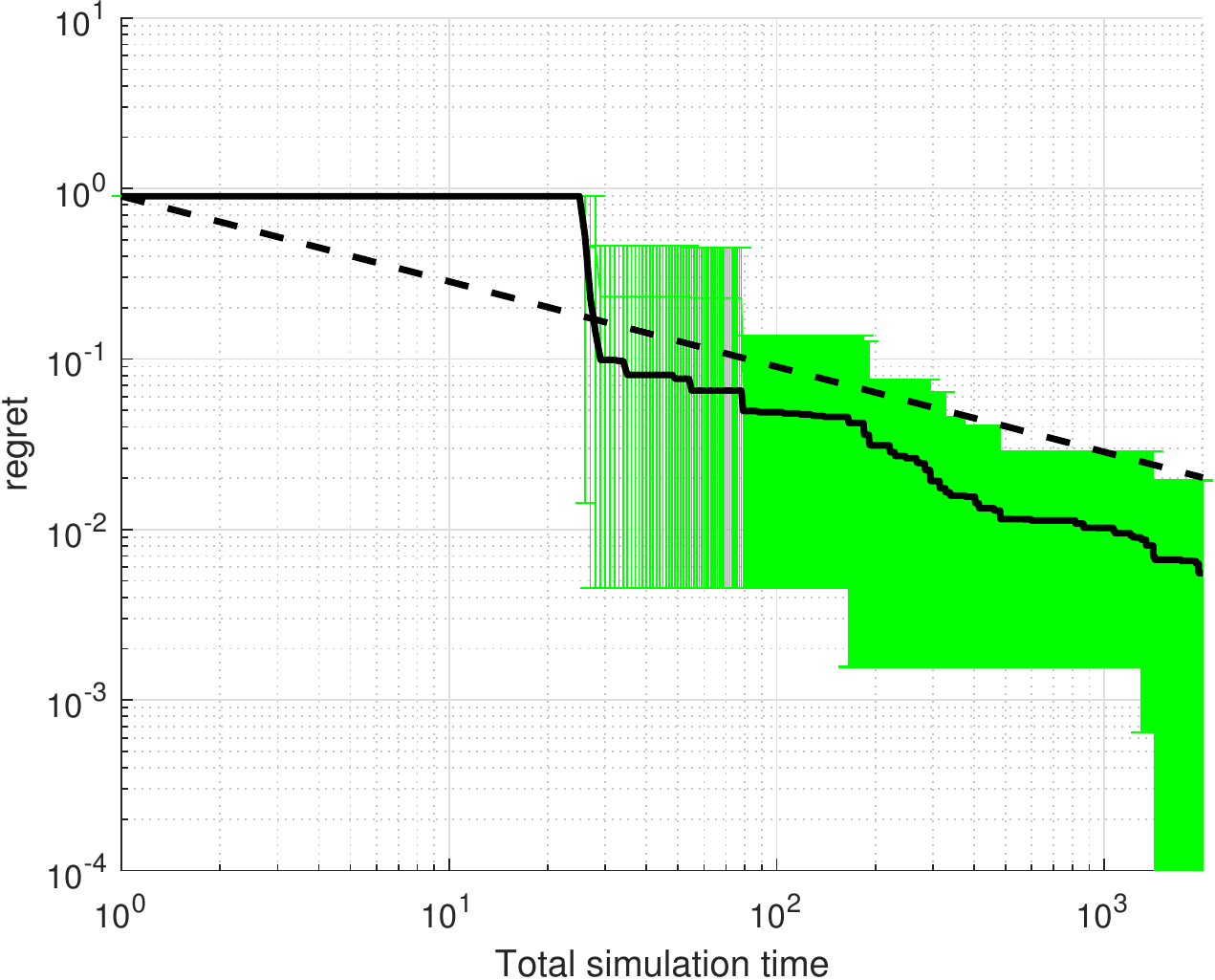}~\vskip-0.01in
    \caption{The regret function in 2D.} 
    \end{subfigure}
~\begin{subfigure}[b]{.5\textwidth}
  \centering
    \includegraphics[width=1.1\columnwidth]{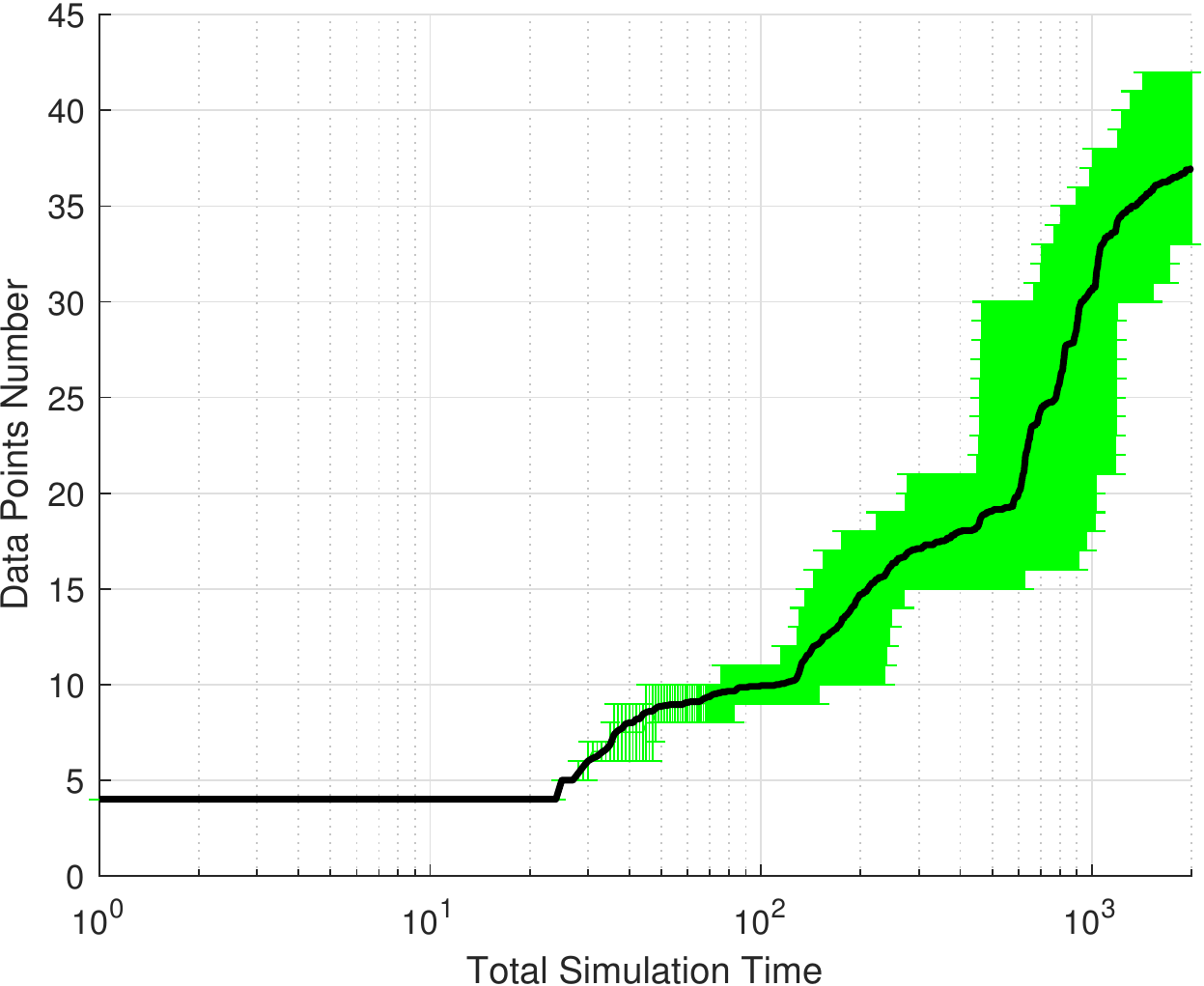}\vskip-0.01in
    \caption{Total number of datapoints in 2D.}  \label{fig:p2d}
     \end{subfigure}
     }
 \centerline{
\begin{subfigure}[b]{.5\textwidth}
    \centering
        \includegraphics[width=1.1\columnwidth]{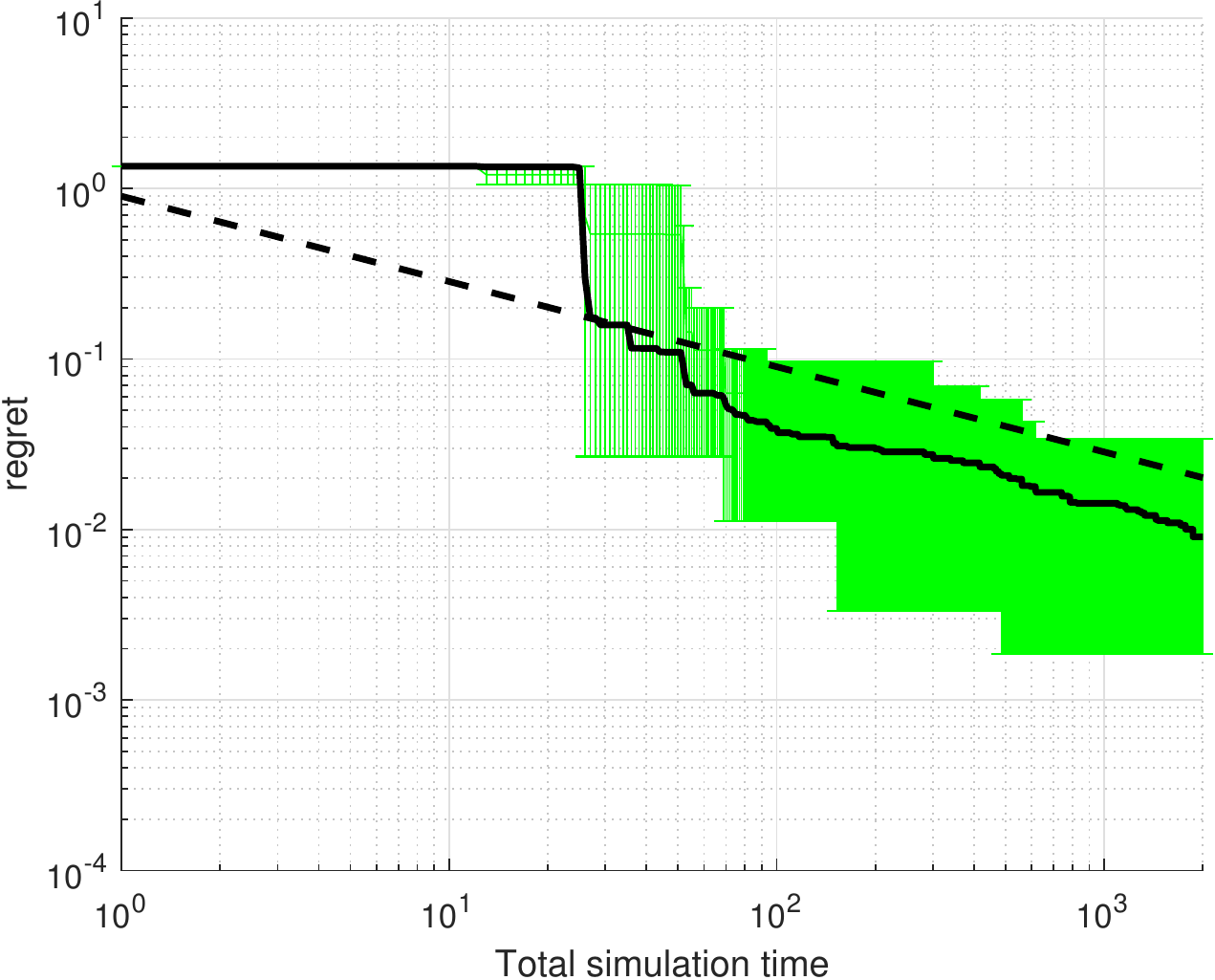}\vskip-0.01in
    \caption{The regret function in 3D.} 
    \end{subfigure}
~\begin{subfigure}[b]{.5\textwidth}
  \centering
    \includegraphics[width=1.1\columnwidth]{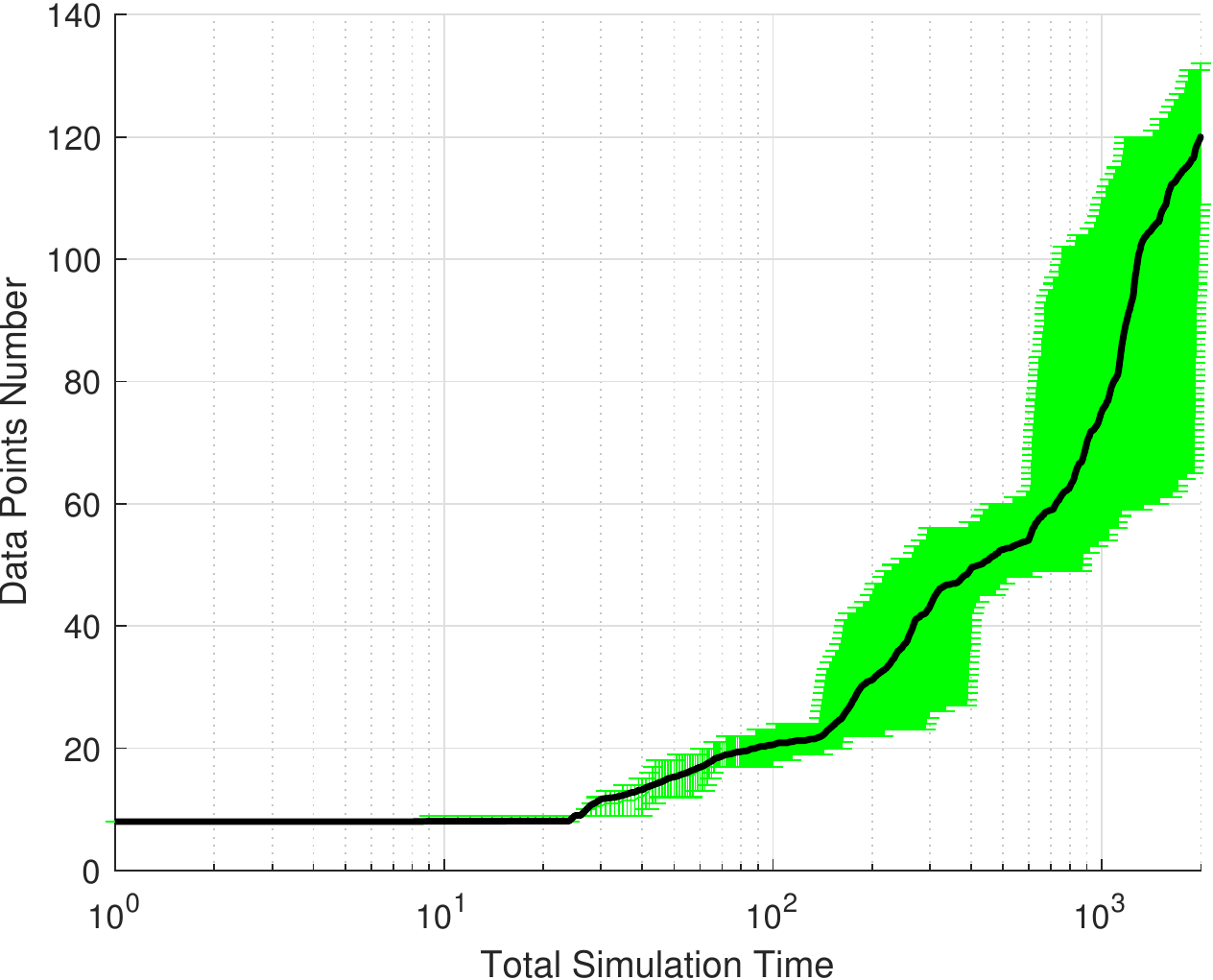}\vskip-0.01in
    \caption{Total number of datapoints in 3D.}  
     \end{subfigure}
     }
   \caption{Implementation of Algorithm \ref{algorithm:alphadogs} on the stochastically-obscured parabolic test problem \eqref{eq:parab}, for twenty different runs.  Left figure shows the mean, min and max value of the regression function over the ensembles. Also plotted at left (dashed bold) is the reference error. }
 \label{fig:parabhis}
\end{figure*}

\begin{figure*}
\centerline{
\begin{subfigure}[b]{.5\textwidth}
    \centering
        \includegraphics[width=1.1\columnwidth]{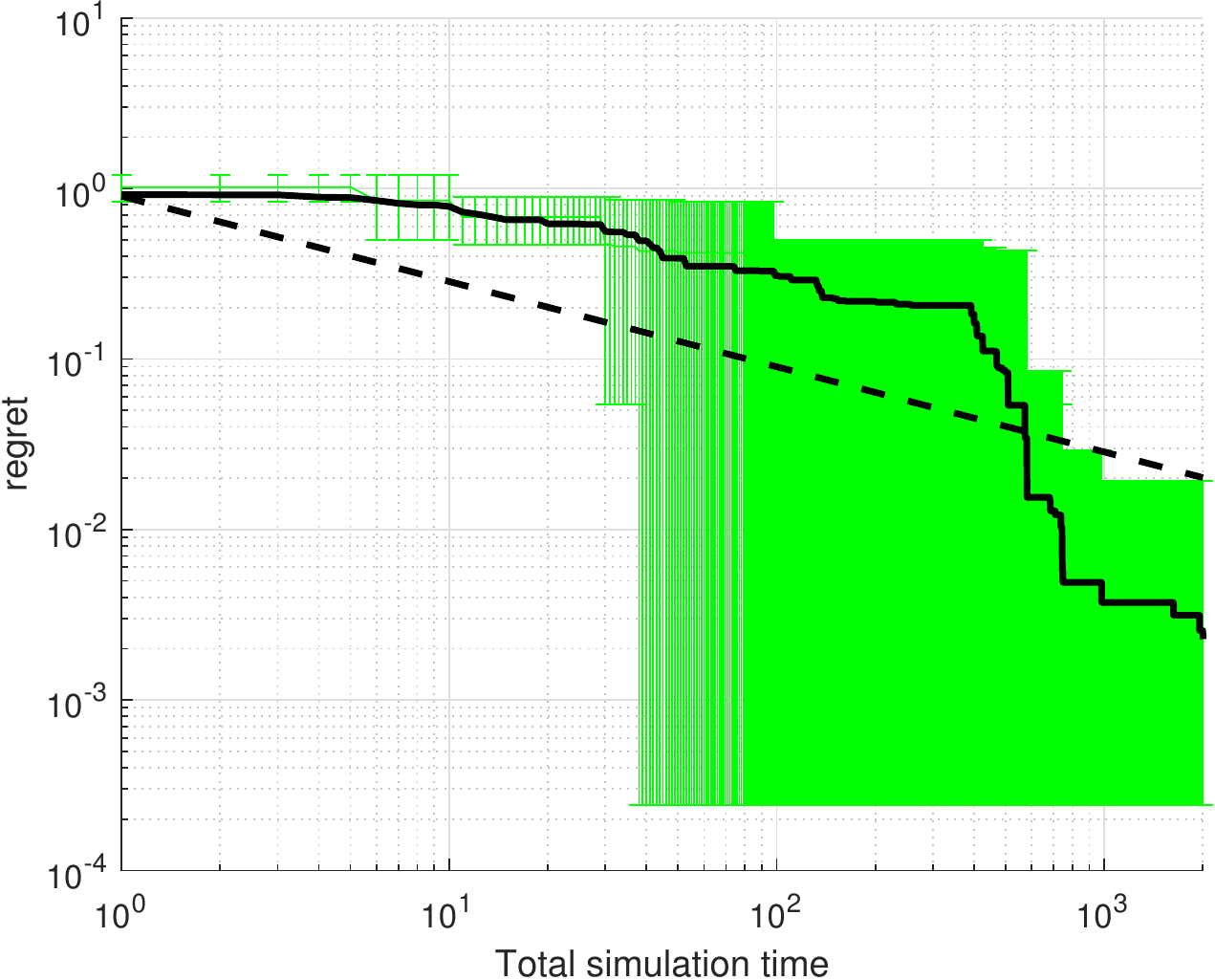}\vskip-0.01in
    \caption{The regret function in 1D.} 
    \end{subfigure}
~\begin{subfigure}[b]{.5\textwidth}
  \centering
    \includegraphics[width=1.1\columnwidth]{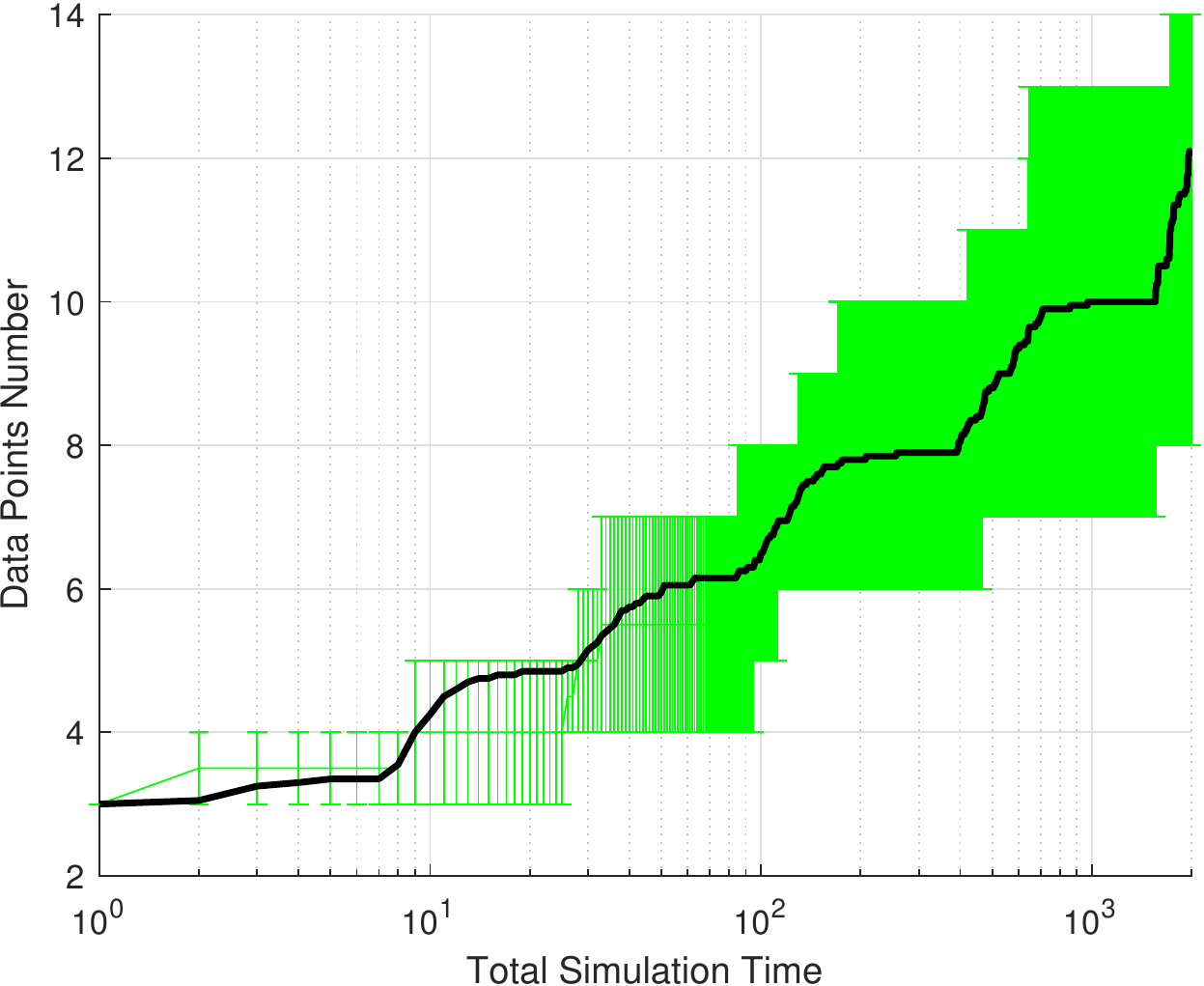}\vskip-0.01in
    \caption{Total number of datapoints in 1D.}  
     \end{subfigure}
     }
 \centerline{
\begin{subfigure}[b]{.5\textwidth}
    \centering
        \includegraphics[width=1.1\columnwidth]{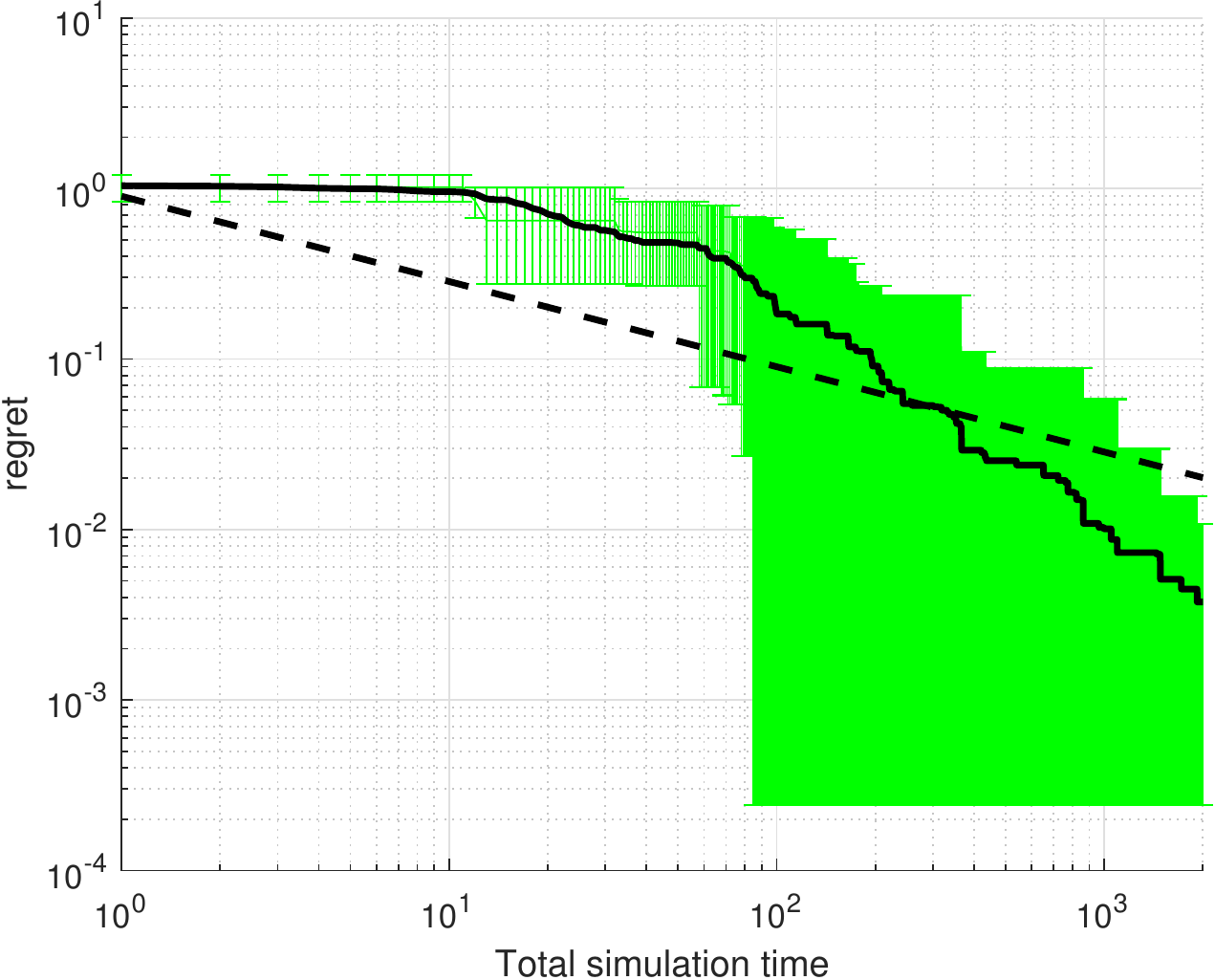}\vskip-0.01in
    \caption{The regret function in 2D.} 
    \end{subfigure}
~\begin{subfigure}[b]{.5\textwidth}
  \centering
    \includegraphics[width=1.1\columnwidth]{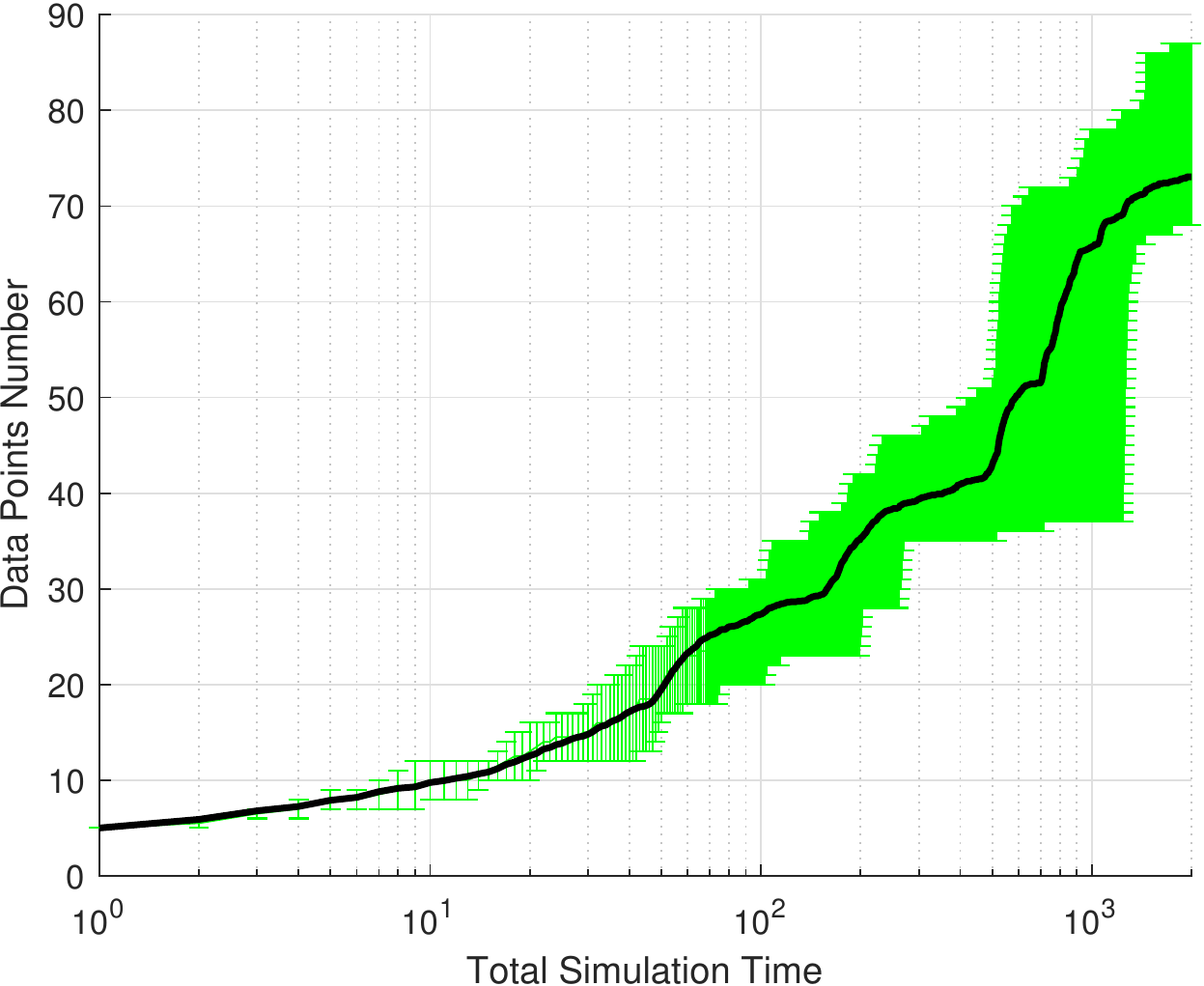}\vskip-0.01in
    \caption{Total number of datapoints in 2D.}  
     \end{subfigure}
     }
 \centerline{
\begin{subfigure}[b]{.5\textwidth}
    \centering
        \includegraphics[width=1.1\columnwidth]{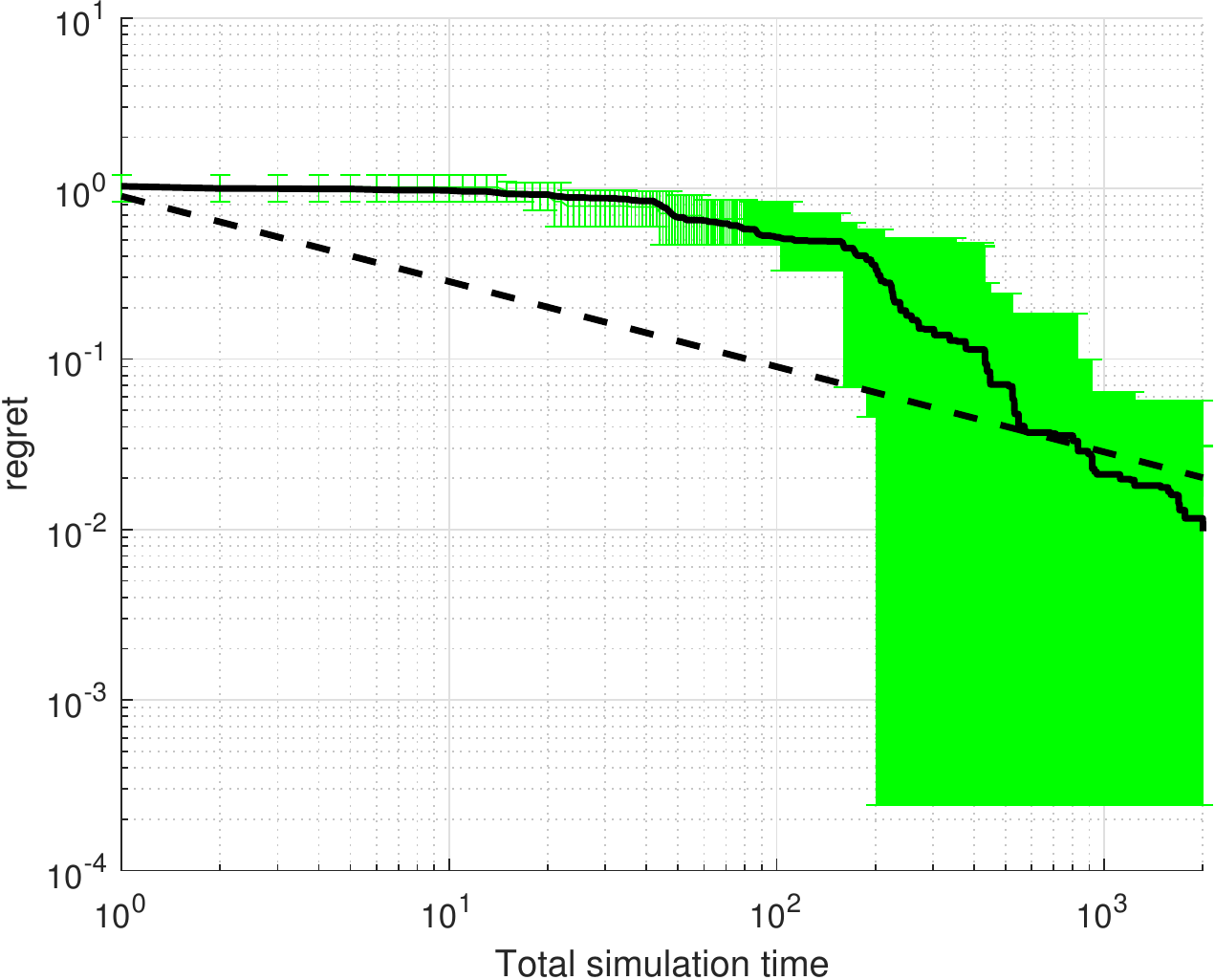}\vskip-0.01in
    \caption{The regret function in 3D.} 
    \end{subfigure}
~\begin{subfigure}[b]{.5\textwidth}
  \centering
    \includegraphics[width=1.1\columnwidth]{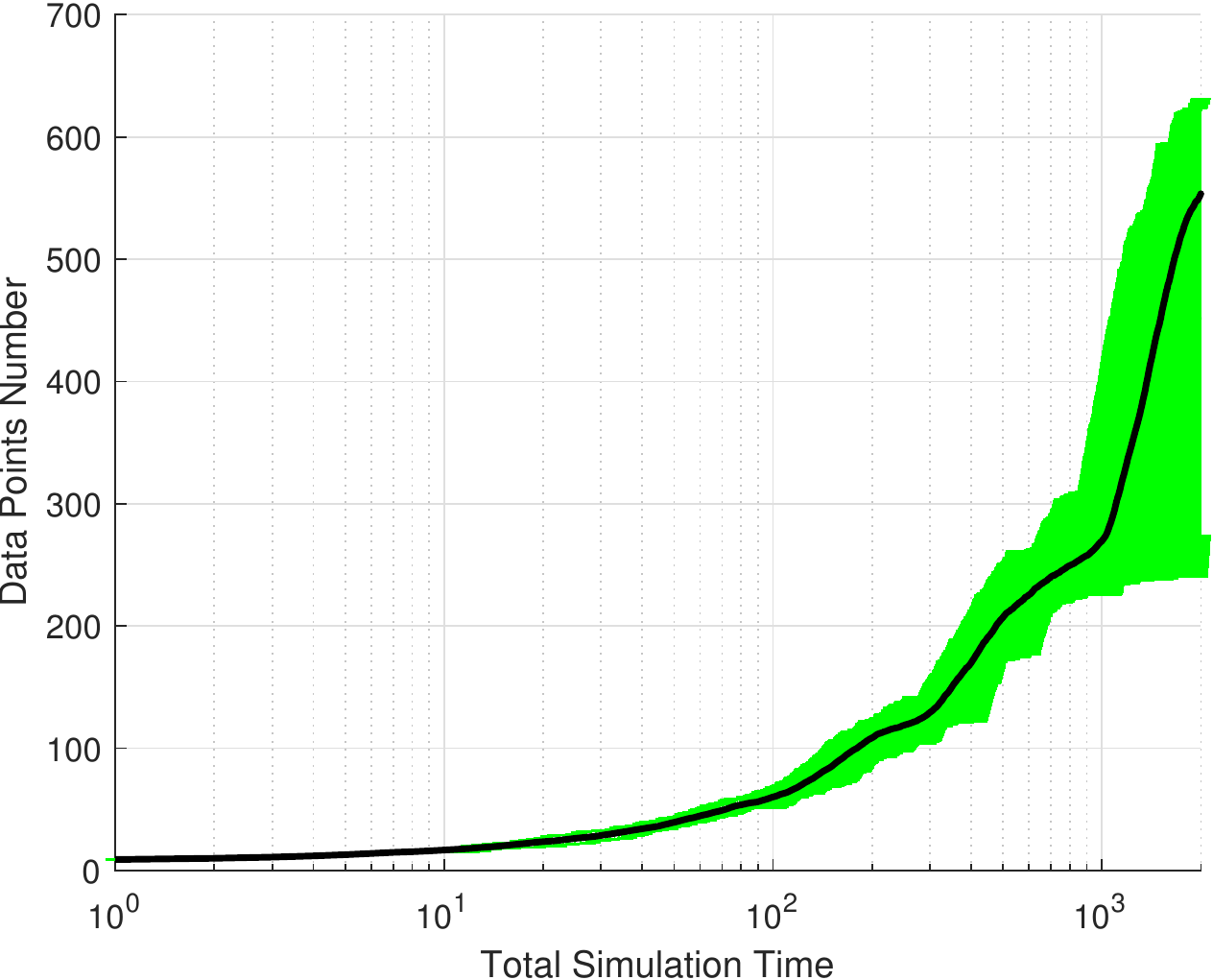}\vskip-0.01in
    \caption{Total number of datapoints in 3D.}  
     \end{subfigure}
     }
   \caption{Implementation of Algorithm \ref{algorithm:alphadogs} on the stochastically-obscured Schwefel test problem \eqref{eq:Schwefel}, for twenty different runs.  Left figure shows the mean, min and max value of the regression function over the ensembles. Also plotted at left (dashed bold) is the reference error.
   }
 \label{fig:shwefelbhis}
\end{figure*}



\section{Application of $\alpha$-DOGS to estimate the parameters of a Lorenz system }\label{sec:result6}

In this section, 
Algorithm \ref{algorithm:deltadogs} ($\Delta$-DOGS),
Algorithm \ref{algorithm:alphadogs} ($\alpha$-DOGS), 
and the Surrogate Management Framework (SMF) developed in \cite{booker-1999} and implemented in \cite{mardsen-2004} are applied to a representative optimization problem, based on infinite-time-averaged statistics, of the type considered in this paper [see \eqref{eq:ProblemStatement}].

The specific problem considered here is derived from the well-known 3-state Lorenz model \cite{david-2016, stewart2000mathematics}, the dynamics of which exhibit a familiar chaotic behavior that roughly characterizes the bulk flow of a fluid within a hollow torus that is heated from below and cooled from above \cite{bewley-2001}, and is governed by
\begin{subequations} \label{eq:lorenzsystem6}
\begin{eqnarray}
  \frac{d}{dt}X &=& s\, (Y-X), \label{eq:X6}\\
  \frac{d}{dt}Y &=& -XZ + \rho\, X -Y, \label{eq:Y6}\\
  \frac{d}{dt} Z &=& XY - \beta\, Z, \label{eq:Z6}
\end{eqnarray}
\end{subequations}
where $(X,Y,Z)$ are the components of the state, which generally moves along a chaotic attractor, and $(\rho,\beta,s)$ are the three (constant) adjustable parameters which affect various characteristics of this attractor. 
The infinite-time-averaged mean and standard deviation of the $Z$ component of this ergotic system are given by:
\begin{subequations}
    \begin{equation}\label{eq:Zavg}
        \bar{Z} = \lim_{T \rightarrow \infty} \frac{1}{T}\int_{t=0}^{T} \, (Z(t)) \, dt,
    \end{equation}
      \begin{equation}\label{eq:Zstd}
        \hat{Z} = \lim_{T \rightarrow \infty} \sqrt{\frac{1}{T} \, \int_{t=0}^{T} \, (Z(t)-\bar{Z})^2 \, dt}.
    \end{equation}
\end{subequations}
In the optimization problem considered in this section, the value of $s=10$ is taken as known, and the parameters $\rho$ and $\beta$ are considered as optimization variables.  Using the method developed in this paper (which is based on successive finite-time simulations), we will seek the values of $\rho$ and $\beta$ which reproduce known values of $\bar{Z}$ and $\hat{Z}$ in the infinite time averaged statistics of the Lorenz system \eqref{eq:lorenzsystem6}.  Towards this end, the cost function considered in this section is
\begin{equation}\label{eq:lorobj}
f(x) = \abs{(\bar{Z} - 23.57)} +\abs {(\hat{Z}- 8.67)}, \quad x = (\rho, \beta).
\end{equation}
Note that the value of $\hat{Z}$ and $\bar{Z}$ are functions of $\{\rho, \beta\}$.
The search domain for $\{\rho, \beta \}$ is taken as $24 \le \rho \le 29.15$ and $1.8 \le \beta \le 4$; note that the Lorenz system \eqref{eq:lorenzsystem6} exhibits a statistically stationary ergodic behavior everywhere 
within this search domain \cite{david-2016}. The optimal solution to this optimization problem is known to be approximately $\rho=28$, $\beta=2.667$.

The cost function given in \eqref{eq:lorobj} might initially appear to be in a slightly different form than that given in \eqref{eq:infinitecost}.  However, it has the same essential structure, in that $f(x)$ can be approximated with increasing accuracy by increasing the sampling (as well as the computational cost) of any given measurement.  As a result, Algorithm \eqref{algorithm:alphadogs} can be applied to this problem directly, given a sufficiently representative uncertainty quantification (UQ) procedure for the finite-time-averaged approximations of the infinite-time-averaged statistics of interest. 
In this work, for the purpose of illustration, we will use the simple UQ approach proposed in Appendix B, which proves to be adequate for our purposes here; improved UQ approaches developed and discussed elsewhere could certainly be used instead.

To numerically simulate the ODE given in \eqref{eq:lorenzsystem6}, 
we use a standard RK4 method with a uniform timestep of $h=0.05$; this approach provides a reasonably small time discretization error~\cite{oliver-2014, david-2016} for this problem.
[Using the same timestep $h$ in all simulations during the optimization process is a drawback of the $\alpha$-DOGS optimization algorithm as developed in this paper; this limitation will be addressed in a future paper, which is currently under development.]


Initial conditions near the attractor are taken for all simulations, and (for simplicity) the first 2600 timesteps (up to $T=13$) are deleted from all simulations, in order to begin time averaging after the system has closely approached the attractor itself.  [The interesting problem of automating the detection of such ``initial transients'', during which the chaotic system approaches the attractor, is also deferred to a future paper.]

Note that all optimizations performed in this section are terminated when 
\begin{gather} 
\abs{\hat{f}(x,T)-f(x^*)} \le  0.04 \quad \textrm{and} \label{eq:teminationconditionmeasure} \\
\sigma(x,T) \le  0.02, \label{eq:teminationconditionsigma}
\end{gather}
where $\hat{f}(x,T)$ and $\sigma(x,T)$ are the estimates and uncertainty at point $x$, and $f(x^*)=0$ is the global minimum of $f(x)$. 
  
According to the procedure developed in Appendix B, $T=2513$ (502600 timesteps) is the minimum simulation time required to achieve the target uncertainty in \eqref{eq:teminationconditionsigma}. As a result, all simulations of $\Delta$-DOGS and SMF use fixed time-averaging lengths of $T= 2513$.  In contrast, the time averaging length used by $\alpha$-DOGS for each datapoint computed during the optimization process is controlled by the optimizer itself, and depends on two parameters:
\begin{itemize}
\item [a.] The averaging length during identifying sampling iterations, which is denoted by $T_0 = 20$ (4000 timesteps). Note that the first $2513$ timesteps are not included in time-averaging process, as they are in the startup transient.
\item[b.] The additional averaging length during supplemental sampling iterations, which is denoted by $T_1 = 7$ (1400 timesteps).    
\end{itemize} 

\noindent The results of applying the three optimization algorithms considered to problem \eqref{eq:lorobj} may  summarized as follows:
\begin{itemize}
\item[a.] $\alpha$-DOGS requires greatly reduced (up to two orders of magnitude) total averaging time as compared with the other two algorithms considered (see Figure \ref{fig.lorConvergence}).
\item[b.] $\Delta$-DOGS uses fewer actual datapoints than $\alpha$-DOGS (see Figure \ref{fig.lorenz_pnts}) to achieve a desired degree of convergence; however, since the time-averaging length at all datapoint is much higher in $\Delta$-DOGS, the actual rate of convergence, in terms of computation time, is greatly improved using $\alpha$-DOGS. 
\item[c.] For $\alpha$-DOGS, the required averaging times at the datapoints which have reduced cost function values are significantly greater than the required averaging times at the datapoints that are far from the desired solution (see figure \ref{fig.yE_T}). 
As a result, the computational cost of the global exploration process during the optimization is significantly reduced using $\alpha$-DOGS. 
\end{itemize}  

\begin{figure} 
\centering
  \includegraphics[scale =0.8]{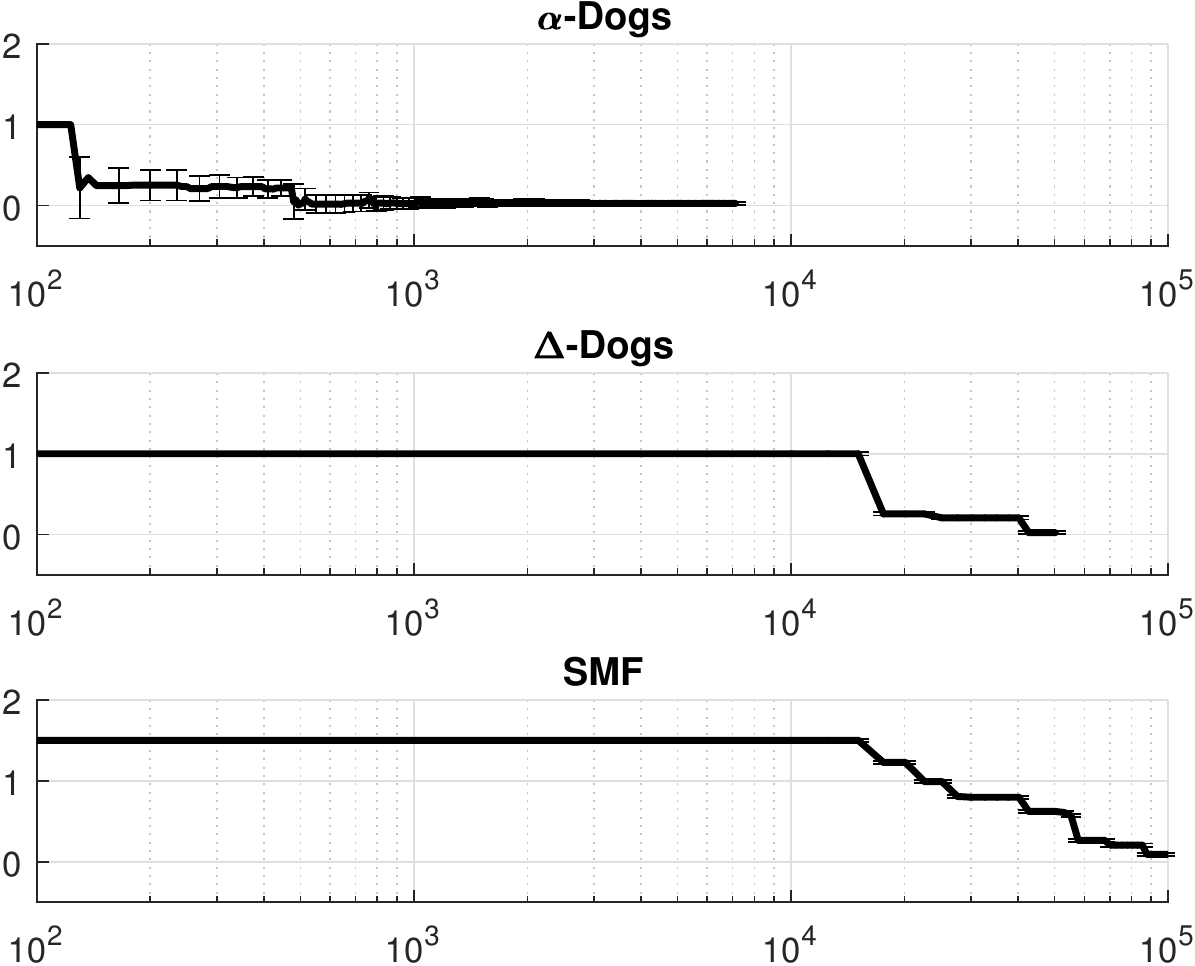}
  \caption{Best measurement vs total averaging length. Convergence history of optimization algorithm for optimization problem \eqref{eq:lorobj} based on Lorenz system.}
  \label{fig.lorConvergence}
\end{figure}
  
\begin{figure}
\centering
  \includegraphics[scale =0.8]{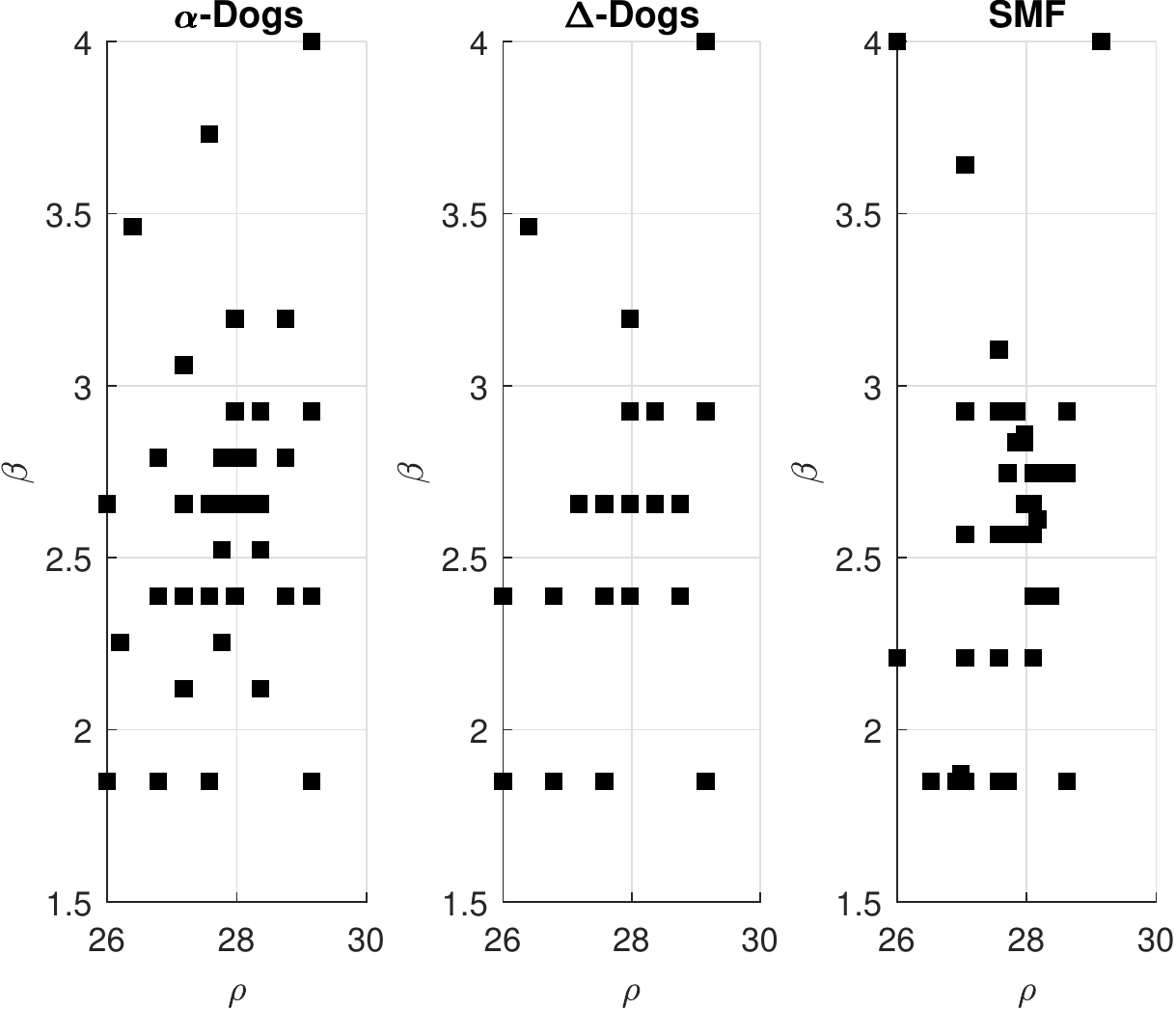}
  \caption{Location of the datapoints considered during the optimization of \eqref{eq:lorobj} based on the Lorenz system.}
  \label{fig.lorenz_pnts}
\end{figure}

\begin{figure}
\centering
  \includegraphics[scale =0.8]{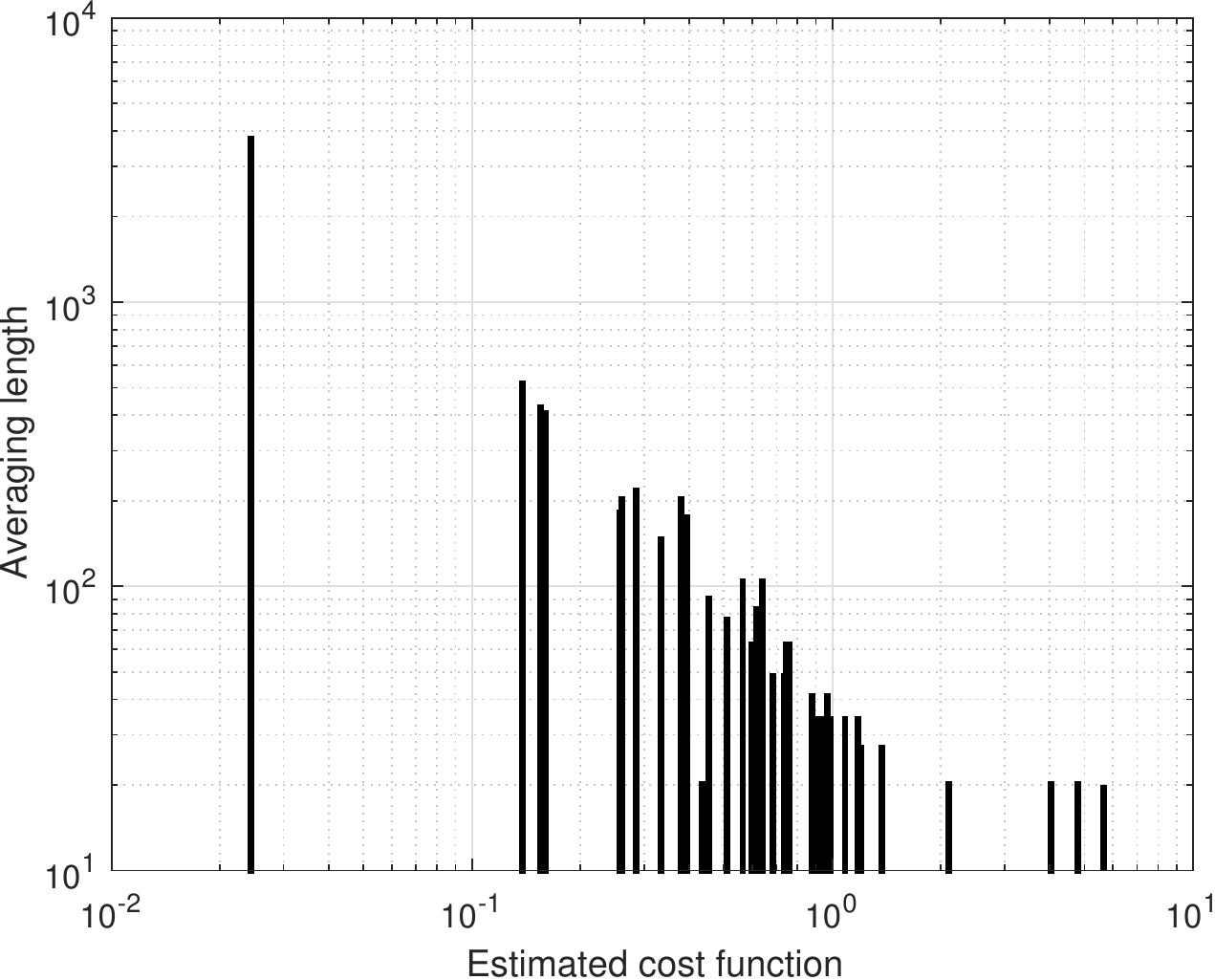}
  \caption{Relationship between the averaging length and cost function value for $\alpha$-DOGS applied on optimization problem \eqref{eq:lorobj}.}
  \label{fig.yE_T}
\end{figure}

\section{Conclusions} \label{sec:conclusion}

This paper presents a new optimization algorithm, dubbed $\alpha$-DOGS, for the minimization of functions given by the infinite-time average of stationary ergodic processes in the computational or experimental setting.
Two search functions are considered at each iteration.  The first is a continuous search functions, $s_c^k(x)$, defined over the entire feasible space $x\in L$, combining a strict regression $p^k(x)$ of the available datapoints together with a remoteness function characterizing the distance of any given point in the feasible domain from the nearest measurements, and built on the framework of a Delaunay triangulation over all available measurements at that iteration.  The second is a discrete search function, $s_d^k(x_i)$, defined over the available measurements $x_i\in S^k$. A comparison between the minima of these two search functions is made in order to decide between further sampling (and, therefore, refining) an existing measurement, or sampling at a new point in parameter space.  The method developed builds closely on the Delaunay-based Derivative-free Optimization via Global Surrogates algorithm, dubbed $\Delta$-DOGS, proposed in \cite{beyhaghi_1,beyhaghi_2,beyhaghi_3}. Convergence of the algorithm is established in problems for which
\begin{itemize}
\item[a.] The underlying truth (infinite-time averaged) function, as well as the regressions computed at each iteration $k$, are twice differentiable.
\item[b.] The stationary process $g(x,k)$ upon which the truth function $f(x)$ is generated, in \eqref{eq:infinitecost}, is ergodic, and the convergence of the averaging process to the underlying truth function is bounded by a monotonic function of a computable uncertainty function (see Assumption \ref{assumption.2}).
\item[c.] The uncertainty of the time averaging process decays exponentially to zero (see Assumption \ref{assumption.3}); this is true for almost all stationary models of random processes.
\end{itemize}

The $\alpha$-DOGS algorithm performs and refines measurements with different amounts of sampling in different locations in the feasible region of parameter space as necessary.  By doing so, the total cost of the optimization process is substantially reduced as compared with using existing derivative-free optimization strategies, with the same amount of sampling at different locations in parameter space.  Computational experiments demonstrate that the algorithm developed ultimately devotes most of its sampling time to points in parameter space near to the global minimum.  Further, these computational experiments indicate that the regret function (see Definition \ref{def:candidatepoint}) eventually diminishes to a value that is actually substantially less than the uncertainty of a single measurement, assuming that all of the sampling is done at a single point.

In future work, the $\alpha$-DOGS algorithm will be applied to additional benchmark and application-based optimization problems, including shape optimization for airfoils and hydrofoils.  For problems in which the function is determined computationally (from, e.g., numerical simulations of turbulent flows), the extension of the present framework to, as convergence is approached, simultaneously
(a) refine the computational grid, and
(b) increase the measurement sampling,
is also under development.

\section*{Appendix A: Polyharmonic spline regression} \label{splineregression}

The algorithm described in this paper depends upon a smooth regression $p^k(x)$ (see Assumption \ref{assumption.1}). 
The best technique for computing the regression is problem dependent. 
As with \cite{beyhaghi_1,beyhaghi_2,beyhaghi_3}, a key advantage of our Delaunay-based approach in the present work is that it facilitates the use of {\it any} suitable regression technique, subject to it satisfying the ``strict'' regression property given in Definition \ref{def:strict}.  Since our numerical tests all implement the polyharmonic spline regression technique, the derivation of this regression technique is briefly explained in this appendix; additional details may be found in \cite{wahba-1990}.

The polyharmonic spline regression $p(x)$ of a function $f(x)$ in $\mathbb{R}^n$ is defined as a weighted sum of a set of radial basis functions $\varphi(r)$ built around the location of each measurement point, plus a linear function of $x$:
\begin{gather}
\label{eq:polyharm}
    p(x) = \sum_{i = 1}^N w_i\,\varphi(r) + v^T
    \begin{bmatrix}
        1 \\
        x
    \end{bmatrix}, \\
    \text{where} \quad \varphi(r) = r^3 \quad \text{and} \quad r = \norm{x - x_i}. \notag
\end{gather}
The weights $w_i$ and $v_i$ represent $N$ and $n+1$ unknowns.
Assume that $\{y(x_1), y(x_2), \\ \dots, y(x_n)\}$ is the set of measurements, with standard deviations $\{\sigma_1, \sigma_2, \dots,\sigma_2 \}$. 
The $w_i$ and $v_i$ coefficients are computed by minimizing the following objective function, which expresses is a tradeoff between the fit to the observed data and
the smoothness of the regressor:
\begin{equation} \label{eq:smoothreg}
L_p(x)= \sum_{i=1}^N \Big[\frac{(p(x_i)-y(x_i))}{\sigma_i}\Big]^2 +\lambda \int_{B} {\abs{\nabla^m p(x)}},
\end{equation}
where $B$ is a large box domain containing all of the $x_i$, and $\nabla^m p(x)$ is the vector including all $m$ derivatives of $p(x)$ (see \cite{duchon-1977}). 
It is shown in \cite{duchon-1977} that the first-order optimality condition for the objective function \eqref{eq:smoothreg} is as follows:
\begin{equation} \label{eq:1stordercond}
p(x_i)-y(x_i)+ \rho \,\sigma_i^2 w_i=0, \quad \forall 1 \le i \le N,  
\end{equation}
where $\rho$ is a parameter proportional to $\lambda$. In summary, the coefficient of the regression can be derived by solving:
    \begin{gather}
        \begin{bmatrix}
            F & V^T \\
            V & 0
        \end{bmatrix}
        \begin{bmatrix}
            w \\
            v
        \end{bmatrix}=
        \begin{bmatrix}
            f(x_i) \\
            0
        \end{bmatrix}, \label{eq:wvrhoeq} \\
        F_{ij} = \varphi(\norm{x_i-x_j}) +\rho \delta_{i,j} \,\sigma_i^2, \quad  \quad 
        V =
        \begin{bmatrix}
            1   & 1   & \dots & 1   \\
            x_1 & x_2 & \dots & x_N
        \end{bmatrix}, \notag
    \end{gather}
where $\delta_{i,j}$ is the Kronecker delta. 

The problem which is left to solve when computing the regression is to find an appropriate value of $\rho \in [0, \infty )$.
Solving \eqref{eq:wvrhoeq} for any value of $\rho$ gives a unique regression, denoted $p(x,\rho)$.  
The parameter $\rho$ is then obtained by a predictive mean-square error criteria developed in \S 4.4 in \cite{wahba-1990}, which is given by imposing the following condition:
\begin{equation}\label{eq:rhoeq}
T(\rho)=\sum_{i=1}^N [\frac{p(x_i,\rho)-y(x_i)}{\sigma_i}]^2=1.
\end{equation}

For $\rho\rightarrow\infty$, $w_i\rightarrow 0$, and the solution of \eqref{eq:wvrhoeq} is a weighted mean-square linear regression,
which is obtained by solving \eqref{eq:rhoeq}. 
If $T(\infty) \le 1$, we take this linear regression as the best current regression for the available data.
Otherwise, we have $T(\infty)>1$ and (by construction) $T(0)=0$; thus, \eqref{eq:rhoeq} has a solution with finite $\rho>0$, which gives the desired regression.

If $T(\infty) > 1$, we thus seek a $\rho$ for which for $T(\rho)=1$. 
Following \cite{wahba-1990}, using \eqref{eq:wvrhoeq}, \eqref{eq:rhoeq} simplifies to:
\begin{equation}
T(\rho)=\rho^2 (\sum_{i=1}^N w_{i}) \,\sigma_i)^2 =1,
\end{equation}\label{eq:Trho}
where $w_{i, \rho}$ is the $w_i$ which is obtained by solving \eqref{eq:wvrhoeq}.
Define $Dw$ and $Dv$ as the vectors whose $i$-th elements are the derivatives of $w_i$ and $v_i$ with respect to $\rho$, then
\begin{gather*}
T'(\rho)= \rho^2 \sum_{i=1}^N w_i \, Dw_i \sigma_i^2+ 2\, \rho (\sum_{i=1}^N w_{i, \rho} \sigma_i)^2, \\
        \begin{bmatrix}
            F & V^T \\
            V & 0
        \end{bmatrix}
        \begin{bmatrix}
            Dw \\
            Dv
        \end{bmatrix}+ 
         \begin{bmatrix}
            \rho \Sigma_2 & 0 \\
            0 & 0
        \end{bmatrix}
        \begin{bmatrix}
            w \\
            v
        \end{bmatrix}=
        \begin{bmatrix}
            0 \\
            0
        \end{bmatrix} \notag,
\end{gather*}
where $\Sigma_2$ is a diagonal matrix whose $i$-the diagonal element is $\rho\,\sigma_i^2$.
Therefore, the analytic expression for the derivative of $T(\rho)$ is available. 
Thus, \eqref{eq:rhoeq} can be solved quickly using Newton's method.  

The regression process presented here, imposing \eqref{eq:Trho} as suggested by \cite{wahba-1990}, is designed to obtain a regression which is reasonably smooth.  However, there is no guarantee that this particular regression satisfies the strictness property required in the present work (see Definition \ref{def:strict}). 
Note, however, that by imposing $\rho=0$, the regression is made strict for arbitrary small $\beta$. 
Thus, to satisfy strictness for a given finite $\beta$, the value of $\rho$ must sometimes be decreased from that which satisfies \eqref{eq:Trho}, as necessary.

\section*{Appendix B: UQ for finite-time-averaging of the Lorenz system} \label{UQ}

This appendix summarizes briefly the simple empirical approach that is used in this paper to quantify the uncertainty of the cost function \eqref{eq:lorobj} when it is estimated using a finite time average.

In the method used, we simply simulated the Lorenz system \eqref{eq:lorenzsystem6} 30 independent times with various initial values for $(X,Y,Z)$.  The cost function was then approximated using these different simulation lengths, and the standard deviation of the estimations obtained using various simulation lengths was calculated. 
The simple model given by ${A}/{\sqrt{T}}$ for the uncertainty was found to fit this empirical calculation quite well, as shown in Figure \ref{fig.UQ}.   

\begin{figure}
\centering
  \includegraphics[scale =0.8]{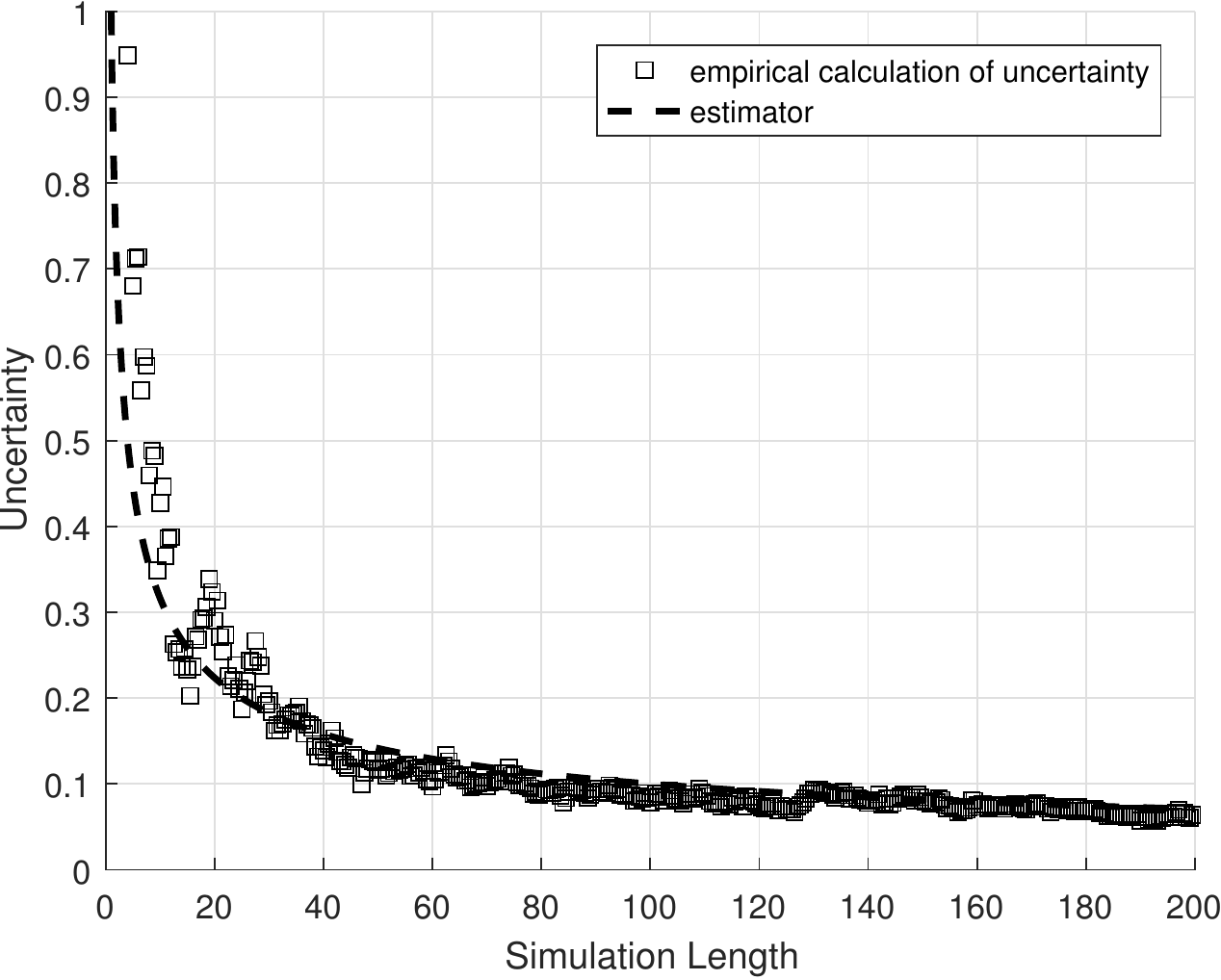}
  \caption{Uncertainty Quantification (UQ) for finite-time-average approximations of the inifinite-time-average statistic of interest in the cost function related to the Lorenz model. 
  The uncertainty quantification model of ${A}/{\sqrt{T}}$ is found to give a very good empirical fit.}
  \label{fig.UQ}
\end{figure}

\section*{Acknowledgment}


We gratefully acknowledge Prof. Phillip Gill and Prof. Alison Marsden for their collaborations and funding from AFOSR FA 9550-12-1-0046, Cymer Center for Control Systems \& Dynamics, and Leidos corporation in support of this work.

\bibliographystyle{plain}
\bibliography{ref}

\end{document}